\documentclass{amsart}

\usepackage[T1]{fontenc}
\usepackage{amssymb}

\usepackage{graphicx}
\usepackage{booktabs}
\usepackage{longtable}
\usepackage{multicol}
\usepackage[ruled,vlined]{algorithm2e}
\usepackage{float}

\usepackage[hidelinks]{hyperref}

\newtheorem{thm}{Theorem}[section]
\newtheorem{lemma}[thm]{Lemma}
\newtheorem{prop}[thm]{Proposition}
\newtheorem{conjecture}[thm]{Conjecture}

\theoremstyle{definition}
\newtheorem{question}[thm]{Question}
\newtheorem{definition}[thm]{Definition}

\theoremstyle{remark}
\newtheorem{remark}[thm]{Remark}
\newtheorem{fact}[thm]{Fact}
\numberwithin{equation}{section}
\newcommand{\tablegap}{\par\vspace{0.2cm}}

\title[Pontryagin-surface boundaries]{Discrete embeddings of hyperbolic
groups with Pontryagin-surface boundaries}

\author{Jiming Ma}
\address{School of Mathematical Sciences, Fudan University, Shanghai, 200433, P. R. China}
\email{majiming@fudan.edu.cn}

\author{Junseo Yoon}
\address{Department of Mathematical Sciences, College of Natural Sciences,
Seoul National University, 1 Gwanak-ro, Gwanak-gu, Seoul 08826,
Republic of Korea}
\email{pi\_lovelove2@snu.ac.kr}

\author{Fangting Zheng}
\address{Department of Mathematical Sciences, Xi'an Jiaotong-Liverpool
University,\newline Suzhou 215123, China}
\email{Fangting.Zheng@xjtlu.edu.cn}

\date{August 13, 2026}
\thanks{Jiming Ma was supported by NSFC No. 12171092. Fangting Zheng was supported by NSFC No. 12471067.}

\begin{document}

\begin{abstract}
Let \(X_p\) be the quotient of the closed disk obtained by identifying
boundary points under rotation through angle \(2\pi/p\).  For every
\(2\leq p\leq8\), we construct a hyperbolic right-angled Coxeter group
with nerve homeomorphic to \(X_p\) that admits a discrete, faithful,
convex cocompact reflection representation into
\(\operatorname{Isom}(\mathbf H^5)\), whose limit set is homeomorphic to
the index-\(p\) Pontryagin surface \(\Pi_p\).  Dimension five is optimal,
since \(\Pi_p\) does not embed in \(\mathbb S^3\).  For \(p=2,3\), the
constructions are analytic and yield cyclically symmetric infinite
families.  %Rigorous interval certification gives asymmetric realizations
%throughout the stated range, while a collar construction gives an
%additional symmetric realization for \(p=4\).  We also construct
%non-right-angled Coxeter reflection representations for \(p=2,3,4\).
\end{abstract}

\maketitle
%\tableofcontents
\section{Introduction}\label{sec:intro}

Nearly a century ago, Pontryagin \cite{Pont:1930} ingeniously constructed, for each
prime \(p\), a compactum \(\Pi_p\) now known as a Pontryagin surface.
These spaces exhibit a striking dimension defect: each \(\Pi_p\) is
two-dimensional, whereas \(\Pi_p\times\Pi_q\) is three-dimensional for
distinct primes \(p\) and \(q\).

We recall the construction briefly; see
\cite{APon:1990,Dran:2011}.  Let \(A\subset\mathbb C\) be the annulus
bounded by the circles \(S_1\) and \(S_2\) of radii \(1\) and \(2\),
respectively.  On \(S_1\), impose the equivalence relation induced by the
degree-\(p\) map \(f_p(z)=z^p\), and let \(M_p\) be the resulting quotient
of \(A\).  This space is called the mod-\(p\) M\"obius band; for \(p=2\),
it is the usual M\"obius band.  Let \(K_p\subset M_p\) denote the image
of \(S_1\).

Let \(Y\) be a triangulated 2-sphere.  In the interior of every
two-simplex \(\Delta\), remove a small open disk and attach a copy of
\(M_p\) along its outer boundary \(S_2\).  Denote the resulting space by
\(M_p(Y)\), and write \(M_p(\Delta)\) for the portion lying over
\(\Delta\).  There is a natural map
\(f^1_0\colon Y_1=M_p(Y)\to Y_0=Y\) whose restriction to each
\(M_p(\Delta)\) fixes \(\partial\Delta\), collapses \(K_p\) to the
barycenter of \(\Delta\), and maps
\(M_p(\Delta)\setminus K_p\) homeomorphically onto the punctured
simplex.  Iterating this operation produces an inverse system
\[
Y_0\xleftarrow{f^1_0}Y_1\xleftarrow{f^2_1}Y_2\xleftarrow{}\cdots .
\]
Its inverse limit is the index-\(p\) Pontryagin surface \(\Pi_p\), whose
homeomorphism type is independent of the auxiliary triangulations.

Although Pontryagin surfaces are traditionally indexed by primes, we use
the same degree-\(p\) inverse-limit construction and the notation
\(\Pi_p\) for every integer \(p\geq2\).  The complexes and Coxeter groups
considered below are likewise defined for arbitrary integers \(p\geq2\);
see Remark~\ref{remark:defP}.

One basic question in geometric group theory is the following.
\begin{question}\label{ques:boundary}
Which topological spaces occur as Gromov boundaries of hyperbolic groups
\cite{KapovichB:2002}?
\end{question}
Beyond spheres, relatively few explicit examples are known.  For
instance, the Menger compactum \(\mu^n\) occurs in dimensions
\(n=0,1,2,3\) \cite{DOsajda:1997,KapovichB:2002}, and certain trees of
manifolds also arise as Gromov boundaries
\cite{PrzytyckiS:2009,Swiatkowski:2009}.  Dranishnikov constructed
hyperbolic Coxeter groups with Pontryagin-surface boundaries
\cite{Dran:1997,Dran:1999}.

A related realization problem in hyperbolic geometry asks:
\begin{question}\label{ques:limitset}
Which topological spaces occur as limit sets of geometrically finite
higher-dimensional Kleinian groups?
\end{question}

A particularly important case is to determine whether a given hyperbolic
group \(G\) admits a discrete, faithful representation into
\(\operatorname{Isom}(\mathbf H^k)\), with \(k\) as small as possible.
This problem remains difficult in general; see, for example,
\cite{Kapovich:2005}.

For \(p\geq2\), let \(X_p\) be the quotient of the closed unit disk
\(\mathbb D^2\) by the boundary identifications
\[
x\sim e^{2\pi i j/p}x,
\qquad x\in\partial\mathbb D^2,\quad 0\leq j\leq p-1.
\]
Thus \(X_p\) is the Moore space \(M(\mathbb Z/p\mathbb Z,1)\), and
\(X_2\cong\mathbb{RP}^2\).  Our main result is the following.

\begin{thm}\label{thm:pontryaginsmallp}
For each \(p\in\{2,3,4,5,6,7,8\}\), there is a hyperbolic right-angled Coxeter
group \(G_p\), whose nerve is a triangulation of \(X_p\), admitting a
discrete, faithful, convex cocompact representation
\[
\rho_p\colon G_p\longrightarrow\operatorname{Isom}(\mathbf H^5).
\]
If \(\Gamma_p=\rho_p(G_p)\), then the limit set of \(\Gamma_p\) is
homeomorphic to the Pontryagin surface \(\Pi_p\).
\end{thm}
Theorem~\ref{thm:pontryaginsmallp} therefore provides new topological
types of limit sets of higher-dimensional Kleinian groups.  The ambient
dimension is optimal.  Indeed, if a prime \(q\) divides \(p\), the usual
Mayer--Vietoris calculation gives
\[
 \dim_{\mathbb Z/q\mathbb Z}\Pi_p=2,
 \qquad
 \dim_{\mathbb Q}\Pi_p=1.
\]
Thus \(\Pi_p\) is not dimensionally full-valued and cannot embed in
\(\mathbb S^3\) by \cite[Theorem~1.10]{Dran:2011}.  Consequently, none of
the spaces occurring in Theorem~\ref{thm:pontryaginsmallp} can be the
limit set of a discrete subgroup of
\(\operatorname{Isom}(\mathbf H^4)\).

\begin{conjecture}\label{conj:all-p}
For every integer \(p\geq2\), there exists a hyperbolic right-angled
Coxeter group \(W_p\), with nerve a triangulation of \(X_p\), that admits
a discrete, faithful, convex cocompact reflection representation
\[
 W_p\longrightarrow\operatorname{Isom}(\mathbf H^5).
\]
Consequently, the image of this representation has limit set homeomorphic
to the Pontryagin surface \(\Pi_p\).
\end{conjecture}

We prove Theorem~\ref{thm:pontryaginsmallp} by establishing
Propositions~\ref{prop:smallp} and~\ref{prop:morep}.
Proposition~\ref{prop:smallp} retains a large cyclic symmetry and gives
especially transparent families for the small indices.

\begin{prop}\label{prop:smallp}
For \(p\geq2\) and \(m\geq5\), the right-angled Coxeter group
\(G_{p,m}\) defined in Subsection~\ref{subsection:group} is Gromov hyperbolic and has boundary homeomorphic to
\(\Pi_p\).  The collar-thickened group \(G_{4,5,3}\) is also Gromov
hyperbolic and has boundary homeomorphic to \(\Pi_4\).  Moreover:
\begin{itemize}
    \item for \(p=2,3\) and every \(m\geq5\), the group \(G_{p,m}\)
    admits a \(\mathbb Z_{pm}\)-symmetric discrete, faithful, convex
    cocompact representation into \(\operatorname{Isom}(\mathbf H^5)\);
    \item the group \(G_{4,5,3}\) defined in
    Subsection~\ref{subsection:collar} admits a
    \(\mathbb Z_{20}\)-symmetric discrete, faithful, convex cocompact
    representation into \(\operatorname{Isom}(\mathbf H^5)\).
\end{itemize}
\end{prop}

\begin{remark}
The groups \(G_{2,m}\), \(m\geq5\), are pairwise non-isomorphic because
their abelianizations have different ranks.  We expect that this family
contains infinitely many commensurability classes, and possibly
infinitely many quasi-isometry classes.
\end{remark}

The full representation variety of \(G_{p,m}\) in
\(\operatorname{Isom}(\mathbf H^5)\) is large.  Our construction instead
uses a four-parameter ansatz with a prescribed
\(\mathbb Z_{pm}\)-symmetry.  This symmetry makes the equations tractable
but also introduces additional rigidity.  The resulting three-layer
ansatz is obstructed for \(p\geq5\), and also for \(p=4\) when
\(m\geq6\); see Remark~\ref{remark:three-layer-obstruction}.  The  remaining
pair \((p,m)=(4,5)\) can also be excluded by a separate argument, and the resulting obstruction
leads naturally to the collar construction of
Subsection~\ref{subsection:collar}.  

A first approach to moving beyond the cyclic ansatz is to retain the right-angled condition while dropping the symmetry assumption. Applying Dranishnikov's special subdivision to suitably chosen asymmetric triangulations of \(X_p\) yields the following:

\begin{prop}\label{prop:morep}
For every \(p\in\{2,3,4,5,6,7,8\}\), the right-angled Coxeter group
\(W_p\) defined in Subsection~\ref{subsection:asysub} is Gromov
hyperbolic, has boundary homeomorphic to \(\Pi_p\), and admits a
discrete, faithful, convex cocompact reflection representation into
\(\operatorname{Isom}(\mathbf H^5)\).
\end{prop}

The upper bound \(p\leq8\) in Proposition~\ref{prop:morep} reflects the
cost of the search-and-certification procedure rather than an observed
mathematical obstruction.  As \(p\) increases, both the size of the
subdivided nerve and the dimension of the corresponding nonlinear system
grow, and the running time required for rigorous interval certification
increases rapidly.  By \(p=8\), however, the underlying triangulations
already exhibit a stable combinatorial pattern.  Verifying further values
one at a time would therefore extend only the finite list of certified
cases, without addressing the general existence problem.  We stop at
\(p=8\) for this reason and regard the observed pattern as evidence for
Conjecture~\ref{conj:all-p}.

Another way to recover flexibility is to retain the cyclic
combinatorics but allow acute Coxeter angles.  This leads to a
two-layer weighted construction and the following complementary result.

\begin{thm}\label{thm:pontryaginlargep}
For each \(p\in\{2,3,4\}\), there is a non-right-angled hyperbolic
Coxeter group \(J_p\) whose weighted nerve is a triangulation of \(X_p\)
and which admits a \(\mathbb Z_{5p}\)-symmetric discrete, faithful, convex
cocompact representation
\[
J_p\longrightarrow\operatorname{Isom}(\mathbf H^5).
\]
\end{thm}

For \(p=2\), the boundary of \(J_p\) is \(\Pi_2\)
\cite{Swiatkowski:2009}.  For \(p=3,4\), we expect the boundary to be
\(\Pi_p\); see Remark~\ref{remark:nonrightangledboundary}.  The weighted
construction also suggests a route toward further topological types of
limit sets in \(\partial\mathbf H^5\).

For comparison, Theorem~4.2 of \cite{HPWalsh} describes every planar
boundary of a hyperbolic group as the limit of a tree system whose
vertex spaces are \(2\)-spheres (only when the underlying tree is
trivial), Sierpi\'nski carpets, circles, disconnected compacta, or
spaces arising recursively from such tree systems.  In particular,
among one-ended convex cocompact subgroups of
\(\operatorname{Isom}(\mathbf H^3)\) with trivial Bowditch JSJ
decomposition, only three homeomorphism types of limit sets can occur.
The situation suggested by our constructions in dimension five is
markedly different.  The Pontryagin surfaces \(\Pi_p\) are pairwise
nonhomeomorphic as \(p\) ranges over the primes,  and are connected without local cut points
\cite{APon:1990, Dran:2011}.  The certified cases established here,
together with the structural pattern underlying
Conjecture~\ref{conj:all-p}, provides evidence that infinitely many homeomorphism types occur among the limit sets of one-ended convex
cocompact subgroups of
\(\operatorname{Isom}(\mathbf H^5)\) with trivial decomposition.
Our results could also be compared with several recent constructions.
Bourdon \cite{Bourdon:1997} exhibited a discrete subgroup of
\(\operatorname{Isom}(\mathbf H^4)\) with universal Menger-curve limit
set.
Douba--Lee--Marquis--Ruffoni \cite{DLMR:2025} produced a convex
cocompact subgroup of \(\operatorname{Isom}(\mathbf H^4)\) whose limit
set is a Pontryagin sphere.  Unlike the surfaces \(\Pi_p\), that space
embeds in \(\mathbb S^3\).  Granier \cite{Granier} obtained a universal
Menger-curve limit set in the complex hyperbolic plane; see also
\cite{MaXie:2022}.  More recently, Douba--Lee--Marquis--Ruffoni
\cite{DLMR:2026} developed a general construction of
higher-dimensional Kleinian groups with two- and three-dimensional
limit sets.  A distinguishing feature of
Theorem~\ref{thm:pontryaginsmallp} is the optimality of the ambient
dimension.

For related convex cocompact representations, see
\cite{Kapovich:2005,DHaglund:2013}.  The groups
constructed here have two-dimensional boundaries and hence virtual
cohomological dimension three.  They therefore provide nontrivial
higher-dimensional examples of convex cocompact Kleinian groups.

\medskip
\noindent\textbf{Outline of the paper.}
We begin with flag-no-square triangulations of \(X_p\), whose
right-angled Coxeter groups have Gromov boundary \(\Pi_p\).  We first
exploit the cyclic symmetry of these triangulations to construct explicit
Lorentzian pole vectors for \(p=2,3\), uniformly in \(m\geq5\).  The case
\(p=4\) is obtained by inserting three collar layers and certifying the
resulting solution with interval arithmetic.  We then retain right angles
but abandon cyclic symmetry: suitable asymmetric nerves yield certified
realizations for every \(2\leq p\leq8\).  A common geometric argument
proves that all of these representations are discrete, faithful, and
convex cocompact.  Section~\ref{subsec:representationlargep} pursues the
complementary direction of retaining the cyclic combinatorics while
allowing acute Coxeter angles.  Finally,
Section~\ref{sec:certification-framework} separates numerical search from
rigorous certification and establishes the criterion used in the finite
computations.  The appendices explain the computational origin of the
uniform ansatz and display a complete asymmetric pole configuration.

\medskip
\noindent\textbf{Acknowledgment.}
The authors thank A.~N.~Dranishnikov for helpful correspondence about
hyperbolic groups with Pontryagin-surface boundaries. The authors are sincerely grateful to Gye-Seon Lee for many helpful discussions and for bringing the second author's related work to the attention of the other authors, thereby facilitating this collaboration.

\section{Preliminaries}\label{sec:preliminaries}
\subsection{Hyperbolic space}

We use the projective and hyperboloid models of five-dimensional
hyperbolic space; see \cite{VinbergS:1993}.  Let
\(\mathbb R^{5,1}\) denote \(\mathbb R^6\) equipped with the bilinear
form of signature \((5,1)\)
\[
 \langle \mathbf z,\mathbf w\rangle
 =\mathbf z^{\mathsf T}L\mathbf w,
 \qquad
 L=\begin{pmatrix}
 I_5&0\\
 0&-1
 \end{pmatrix}.
\]
Set
\[
\begin{aligned}
V_-&=\{\mathbf z\in\mathbb R^6\setminus\{0\}:
        \langle\mathbf z,\mathbf z\rangle<0\},\\
V_0&=\{\mathbf z\in\mathbb R^6\setminus\{0\}:
        \langle\mathbf z,\mathbf z\rangle=0\},\\
V_+&=\{\mathbf z\in\mathbb R^6\setminus\{0\}:
        \langle\mathbf z,\mathbf z\rangle>0\}.
\end{aligned}
\]

Let
\[
 [\,\cdot\,]\colon\mathbb R^6\setminus\{0\}\longrightarrow\mathbb{RP}^5
\]
be the canonical projection.  The projective model of
\(\mathbf H^5\) is \([V_-]\), and its ideal boundary is
\(\partial\mathbf H^5=[V_0]\).  We use boldface letters for vectors in
\(\mathbb R^{5,1}\) and the corresponding plain letters for their
projective classes.  The affine chart
 \begin{equation}\label{projection}
 \left(\begin{matrix} z_1 \\ z_2\\ \vdots \\ z_{5} \end{matrix}\right)
 \longmapsto \left[\begin{matrix} z_1 \\
 z_2 \\ \vdots \\ z_5 \\ 1  \end{matrix}\right].
 \end{equation}
identifies \(\mathbf H^5\) with the unit ball
\(\mathbb B^5\subset\mathbb R^5\) and
\(\partial\mathbf H^5\) with \(\mathbb S^4=\partial\mathbb B^5\).

For computations, we also use the hyperboloid model
\[
 \mathbf H^5=
 \{v\in\mathbb R^{5,1}:\langle v,v\rangle=-1,\ v_6>0\}.
\]

Let
\[
\mathbf O(5,1)=
\{A\in\operatorname{GL}(6,\mathbb R):
  \langle A\mathbf x,A\mathbf y\rangle
  =\langle\mathbf x,\mathbf y\rangle
  \text{ for all }\mathbf x,\mathbf y\in\mathbb R^6\}.
\]
The projective orthogonal group
\[
 \mathbf{PO}(5,1)=\mathbf O(5,1)/\{\pm I\}
\]
is naturally identified with \(\operatorname{Isom}(\mathbf H^5)\).

\subsection[The hyperbolic group Gpm]
{The hyperbolic group \texorpdfstring{\(G_{p,m}\)}{Gpm}}
 \label{subsection:group}

Recall that \(X_p\) is obtained from \(\mathbb D^2\) by identifying
boundary points in the same orbit of the rotation through \(2\pi/p\).
Equivalently, \(X_p\) is obtained by attaching a disk to the boundary of
the mod-\(p\) M\"obius band \(M_p\).  It is the Moore space
\(M(\mathbb Z/p\mathbb Z,1)\); in particular,
\(X_2\cong\mathbb{RP}^2\).

A simplicial complex is \emph{flag} if every finite clique in its
one-skeleton spans a simplex.  It is \emph{no-square} if every embedded
four-cycle in its one-skeleton has a diagonal.  A triangulation with
both properties is called \emph{flag-no-square}.

We use the symmetric triangulation
\(\mathcal T_{p,m}\) of \(X_p\) shown in
Figure~\ref{figure:triangulation}.  It admits a
\(\mathbb Z_{pm}\)-symmetry and has the following vertices:
	\begin{enumerate}
	\item the outer red circle contains \(pm\) positions whose labels are
	periodic with period \(m\), so that \(A_{i+m}=A_i\);

		\item the middle blue circle contains the cyclically ordered
		vertices \(\{B_i\}_{i=1}^{pm}\);

	\item the small green circle contains the cyclically ordered
	vertices \(\{C_i\}_{i=1}^{pm}\);

	\item the central vertex is labeled \(O_c\).
	\end{enumerate}

Figure~\ref{figure:triangulation} depicts \(\mathcal T_{3,6}\).  Only
one-third is drawn; the remaining two-thirds are obtained by the
\(\mathbb Z_3\)-symmetry.  The layer-by-layer adjacency rules show
immediately that every clique is a face, every embedded four-cycle has a
diagonal, and the outer identification is the degree-\(p\) attaching map.
Thus:

\begin{prop}\label{prop:triangulation}
For all integers \(m\geq5\) and \(p\geq2\), the complex
\(\mathcal T_{p,m}\) is a flag-no-square triangulation of \(X_p\).
\end{prop}

Notably, for \(m=3,4\), the triangulation \(\mathcal T_{p,m}\) is not
flag-no-square. Given a triangulation \(\mathcal T\) of a two-complex, its
right-angled Coxeter group has one involutory generator for every vertex
of \(\mathcal T\), with two generators commuting exactly when the
corresponding vertices are joined by an edge.

Let \(G_{p,m}\) be the right-angled Coxeter group with nerve
\(\mathcal T_{p,m}\).  For example, \(G_{2,6}\) has:
	 	\begin{enumerate}

		\item \(6+2\cdot6+2\cdot6+1=31\) involutory generators,
		namely
		\[
		\{A_1,\ldots,A_6,B_1,\ldots,B_{12},
		  C_1,\ldots,C_{12},O_c\};
		\]
			\item \(6+24+12+12+24+12=90\) commutation relations:
			\begin{enumerate}
		\item \(A_iA_{i+1}=A_{i+1}A_i\) for \(1\leq i\leq6\), where \(A_7=A_1\);

			\item \(B_jB_{j+1}=B_{j+1}B_j\) and
			\(C_jC_{j+1}=C_{j+1}C_j\) for \(1\leq j\leq12\),
			where \(B_{13}=B_1\) and \(C_{13}=C_1\);

				\item \(A_iB_i=B_iA_i\) and
				\(A_iB_{i+1}=B_{i+1}A_i\) for \(1\leq i\leq6\);

				\item \(A_iB_{i+6}=B_{i+6}A_i\) and
				\(A_iB_{i+7}=B_{i+7}A_i\) for \(1\leq i\leq6\);

			\item \(B_iC_i=C_iB_i\) and
			\(B_iC_{i+1}=C_{i+1}B_i\) for \(1\leq i\leq12\);

				\item \(O_cC_i=C_iO_c\) for \(1\leq i\leq12\).
			\end{enumerate}
	\end{enumerate}

Similarly, \(G_{3,6}\) has
\(6+3\cdot6+3\cdot6+1=43\) generators and
\(6+36+12+12+12+36+18=132\) commutation relations.

A right-angled Coxeter group is Gromov hyperbolic precisely when its
nerve is no-square \cite{Mou}.  Proposition~\ref{prop:triangulation}
therefore implies that \(G_{p,m}\) is hyperbolic for \(m\geq5\) and
\(p\geq2\).  The boundary identification follows from
\cite{Dran:1997,Dran:1999}.
\begin{prop}\label{prop:group}
For \(m\geq5\) and \(p\geq2\), the Gromov boundary of \(G_{p,m}\) is
homeomorphic to the Pontryagin surface \(\Pi_p\).
\end{prop}

\begin{remark}\label{remark:defP}
Dranishnikov states the boundary identification in
Proposition~\ref{prop:group} for prime \(p\)
\cite{Dran:1997,Dran:1999}.  However, primality is not used in the
proof: the argument depends only on the degree-\(p\) attaching map and
therefore applies verbatim to every integer \(p\geq2\).  For \(p=2\),
one may alternatively invoke the tree-of-manifolds theorem of
\'{S}wi\k{a}tkowski \cite{Swiatkowski:2009}, because
\(X_2\cong\mathbb{RP}^2\).  This alternative does not apply when
\(p>2\), since the Moore space \(X_p\) is then not a manifold.
\end{remark}

\begin{figure}[h]
	\centering
    \includegraphics[width=0.8\linewidth]{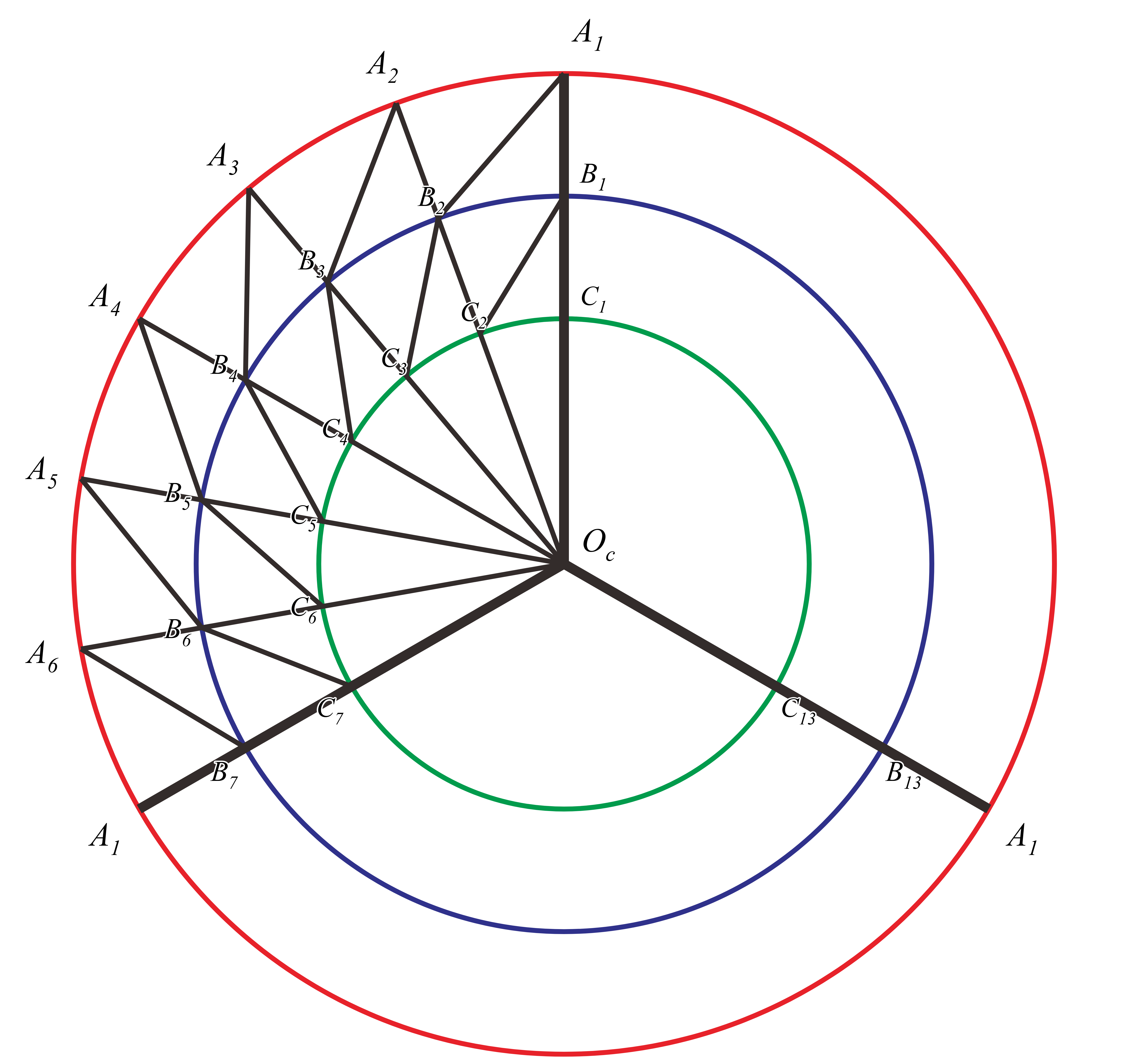}
    \caption{The flag-no-square triangulation
    \(\mathcal T_{3,6}\) of \(X_3\).}
	\label{figure:triangulation}
\end{figure}

\section{Proof of Theorem~\ref{thm:pontryaginsmallp}}
\label{sec:proofsmallp}

\subsection{The cyclic ansatz}
\label{subsec:construction}

Fix integers \(m\geq5\) and \(p\geq2\).  We seek a representation of
\(G_{p,m}\) in \(\mathbf{PO}(5,1)\) that sends each generator
\(A_i,B_j,C_k,O_c\) to reflection in a hyperplane of \(\mathbf H^5\).

We impose the \(\mathbb Z_{pm}\)-symmetry of the nerve.  Let
\[
S=\begin{pmatrix}
\cos(\frac{2  \pi}{ m}) & \sin(\frac{2 \pi}{ m}) &0 & 0    & 0 & 0\\[2 ex]
- \sin(\frac{2  \pi}{ m}) & \cos(\frac{2  \pi}{ m}) & 0 &  0& 0&0\\[2 ex]
0&0& \cos(\frac{2\pi}{p\cdot m}) & \sin(\frac{2\pi}{p\cdot m}) &0&0\\[2 ex]
0&0&- \sin(\frac{2 \pi}{p\cdot m}) &\cos(\frac{2\pi}{p\cdot m}) &0&0\\[2 ex]
0&0&0&0 &
1&0\\[2 ex]
0 & 0 &0 &0 &0&1\\
\end{pmatrix}.
\]
Then \(S\in\mathbf O(5,1)\) has order \(pm\).

Choose initial pole vectors in \(\mathbb R^{5,1}\) of the form
\[
\begin{aligned}
n(O_c)&=(0,0,0,0,1/c,1)^{\mathsf T},\\
n(A_1)&=(x_a,0,0,0,a,1)^{\mathsf T},\\
n(B_1)&=(x_b,y_b,z_b,0,b,1)^{\mathsf T},\\
n(C_1)&=(x_c,y_c,z_c,w_c,c,1)^{\mathsf T}.
\end{aligned}
\]
Thus there are eleven scalar parameters to determine.  For
\(i\in\mathbb Z\), define
\[
n(A_i)=S^{i-1}n(A_1),\qquad
n(B_i)=S^{i-1}n(B_1),\qquad
n(C_i)=S^{i-1}n(C_1).
\]
These vectors satisfy
\[
\begin{aligned}
Sn(O_c)&=n(O_c),& S^mn(A_1)&=n(A_1),\\
S^{pm}n(B_1)&=n(B_1),& S^{pm}n(C_1)&=n(C_1).
\end{aligned}
\]

For two hyperplanes \(\mathcal P_1,\mathcal P_2\subset\mathbf H^5\),
choose spacelike pole vectors \(n_1,n_2\).  The scale-invariant quantity
\begin{equation}\label{eq:Hfunction}
H(n_1,n_2)=
\frac{\langle n_1,n_2\rangle^2}
{\langle n_1,n_1\rangle\langle n_2,n_2\rangle}
\end{equation}
satisfies:
\begin{itemize}
\item \(\mathcal P_1\) and \(\mathcal P_2\) meet at angle \(\theta\) if
and only if \(H(n_1,n_2)=\cos^2\theta\);
\item \(\mathcal P_1\) and \(\mathcal P_2\) are hyperparallel if and only
if \(H(n_1,n_2)>1\).
\end{itemize}

Let \(A_i\) denote reflection in the hyperplane with pole \(n(A_i)\),
and define \(B_j,C_k\), and \(O_c\) similarly.  The prescribed
right-angled edge equations are
\begin{equation}\label{eq:edge-orthogonality}
\begin{aligned}
\langle n(A_1),n(A_2)\rangle
 &=\langle n(A_1),n(B_1)\rangle
  =\langle n(A_1),n(B_2)\rangle
  =\langle n(B_1),n(B_2)\rangle=0,\\
\langle n(B_1),n(C_1)\rangle
 &=\langle n(B_1),n(C_2)\rangle
  =\langle n(C_1),n(C_2)\rangle=0,\\
\langle n(C_1),n(O_c)\rangle&=0.
\end{aligned}
\end{equation}

\subsection[Uniform representations for m at least 5]
{Uniform representations for \texorpdfstring{\(m\geq5\)}{m at least 5}}
\label{subsec:uniform}
We now give a uniform closed-form construction for \(p=2,3\) and every
integer \(m\geq5\).  The principal point is not merely to solve the edge
equations, but to obtain estimates that are uniform in \(m\) and separate
every non-edge pair.  The case \(p=4\) will be treated separately by the
collar construction below.

\noindent For \(p=2\), set
\[
\rho_p=\frac1{16},\qquad \sigma_p=\frac1{16}.
\]
For \(p=3\), set
\[
\rho_p=\frac{17}{100},\qquad \sigma_p=\frac{13}{50}.
\]
Put \(t=2\pi/m\), and define
\begin{equation}\label{eq:R}
R_p(m)=\frac{(1-\rho_pt^2)(1+\cos t)}{2\cos t}.
\end{equation}
We shall also use the following two scalars:
\begin{equation}\label{eq:Lambda}
\begin{aligned}
\Lambda_p(m)
&=
\cos\frac t2\sqrt{(1-\rho_pt^2)(1-\sigma_pt^2)} \\
&\quad+
\frac{\cos(t/(2p))}{\cos(t/p)}
\sqrt{\bigl(1-\cos t(1-\rho_pt^2)\bigr)
      \bigl(1-\cos t(1-\sigma_pt^2)\bigr)}.
\end{aligned}
\end{equation}
Next define
\begin{equation}\label{eq:calM}
\mathcal M_p(m)
=\sqrt{R_p(m)}\,\Lambda_p(m)
 +\sqrt{R_p(m)-1}\sqrt{\Lambda_p(m)^2-1}.
\end{equation}
The eleven variables are assigned as follows:
\begin{equation}\label{eq:a-value}
a=0.
\end{equation}
\begin{equation}\label{eq:b-value}
b=
\sqrt{1-R_p(m)^{-1}},
\end{equation}
\begin{equation}\label{eq:c-value}
c=\sqrt{1-\mathcal M_p(m)^{-2}},
\end{equation}
\begin{equation}\label{eq:xa-value}
x_a=\frac{1}{\sqrt{\cos t}},
\end{equation}
\begin{equation}\label{eq:xb-value}
x_b=\sqrt{\cos t},
\end{equation}
\begin{equation}\label{eq:yb-value}
y_b=\sqrt{\frac{2\cos t}{1+\cos t}}\sin\frac t2,
\end{equation}
\begin{equation}\label{eq:zb-value}
z_b=
\sqrt{
R_p(m)^{-1}
\frac{1-\cos t(1-\rho_pt^2)}{\cos(t/p)}
},
\end{equation}
\begin{equation}\label{eq:xc-value}
x_c=
\frac{\sqrt{1-\sigma_pt^2}}{\mathcal M_p(m)}\cos t,
\end{equation}
\begin{equation}\label{eq:yc-value}
y_c=
\frac{\sqrt{1-\sigma_pt^2}}{\mathcal M_p(m)}\sin t,
\end{equation}
\begin{equation}\label{eq:zc-value}
z_c=
\frac1{\mathcal M_p(m)}
\sqrt{\frac{1-\cos t(1-\sigma_pt^2)}{\cos(t/p)}}
\cos\frac{t}{2p},
\end{equation}
and
\begin{equation}\label{eq:wc-value}
w_c=
\frac1{\mathcal M_p(m)}
\sqrt{\frac{1-\cos t(1-\sigma_pt^2)}{\cos(t/p)}}
\sin\frac{t}{2p}.
\end{equation}

\begin{prop}\label{prop:uniform-smallp}
	For \(p\in\{2,3\}\) and every \(m\geq5\), the vectors defined by
	\eqref{eq:a-value}--\eqref{eq:wc-value} form a solution of the symmetric three-layer system:
	they are spacelike, all prescribed edge inner products vanish, and every
	non-edge pair has negative Lorentzian inner product and satisfies \(H>1\).
\end{prop}

\begin{proof}
	The construction is designed so that the edge equations are identities.

	First,
	\[
	\langle n(A_1),n(A_2)\rangle
	=
	\cos\frac{2\pi}{m}\,x_a^2-1
	=
	0
	\]
	by \eqref{eq:xa-value}.  Also,
	\[
	x_b=x_a^{-1}.
	\]
	Since
	\[
	x_b^2+y_b^2
	=
	\frac{2\cos(2\pi/m)}{1+\cos(2\pi/m)},
	\]
	 one obtains
	\[
	\langle n(A_1),n(B_1)\rangle=
	\langle n(A_1),n(B_2)\rangle=0.
	\]

	The equation
	\[
	\langle n(B_1),n(B_2)\rangle=0
	\]
	is equivalent to
	\[
	\cos\frac{2\pi}{m}\,
	\frac{x_b^2+y_b^2}{1-b^2}
	+
	\cos\frac{2\pi}{mp}\,
	\frac{z_b^2}{1-b^2}
	=
	1.
	\]
	Substituting \eqref{eq:b-value}, \eqref{eq:yb-value}, and \eqref{eq:zb-value} gives an identity.

	Likewise, \(\langle n(C_1),n(C_2)\rangle=0\) is equivalent to
	\[
	\cos\frac{2\pi}{m}\,
	\frac{x_c^2+y_c^2}{1-c^2}
	+
	\cos\frac{2\pi}{mp}\,
	\frac{z_c^2+w_c^2}{1-c^2}
	=
	1,
	\]
	which follows immediately from \eqref{eq:c-value}, \eqref{eq:xc-value}--\eqref{eq:wc-value}.

	It remains to check the two \(B\)-\(C\) edge equations.  By
	\eqref{eq:R}--\eqref{eq:c-value},
	\begin{equation}\label{eq:edge-lambda}
\frac{1-bc}{\sqrt{(1-b^2)(1-c^2)}}=\Lambda_p(m).
\end{equation}
	On the other hand, the Euclidean inner product of the first four coordinates of
	\(n(B_1)\) and \(n(C_1)\), divided by \(\sqrt{(1-b^2)(1-c^2)}\), is exactly
	\(\Lambda_p(m)\).  Hence
	\[
	\langle n(B_1),n(C_1)\rangle=0.
	\]
	After applying \(S\) to \(n(C_1)\), the relevant angle differences become
	\(-\pi/m\) in the first rotation plane and \(-\pi/(mp)\) in the second rotation
	plane; the cosines are unchanged.  Therefore
	\[
	\langle n(B_1),n(C_2)\rangle=0.
	\]
	Finally,
	\[
	\langle n(C_1),n(O_c)\rangle=c\cdot\frac1c-1=0.
	\]

	The well-definedness of the formulas, the spacelike inequalities, and the
	non-edge estimates are supplied by Lemma~\ref{lem:uniform-separation}
	below.  This completes the proof.
\end{proof}

For use in the uniform estimate, put
\[
\begin{aligned}
Q_A&:=\langle n(A_1),n(A_1)\rangle,
&Q_B&:=\langle n(B_1),n(B_1)\rangle,\\
Q_C&:=\langle n(C_1),n(C_1)\rangle,
&Q_{O_c}&:=\langle n(O_c),n(O_c)\rangle=\frac1{c^2}-1.
\end{aligned}
\]
By symmetry, the non-edge quantities reduce to the following eight
families.

	For \(2\le k\le m-2\), the \(A\)-\(A\) non-edge values are
	\begin{equation}\label{eq:HAA}
H_{AA}(k)
	=
	\frac{
		\left(
		x_a^2\cos\frac{2\pi k}{m}+a^2-1
		\right)^2
	}{
		Q_A^2
	}.
\end{equation}
	For \(k\in\{0,\dots,m-1\}\setminus\{0,1\}\), the \(A\)-\(B\) non-edge values are
	\begin{equation}\label{eq:HAB}
H_{AB}(k)
	=
	\frac{
		\left(
		x_a\sqrt{x_b^2+y_b^2}
		\cos\left(\frac{\pi}{m}-\frac{2\pi k}{m}\right)
		+ab-1
		\right)^2
	}{
		Q_AQ_B
	}.
\end{equation}
	For \(0\le k\le m-1\), every \(A\)-\(C\) pair is a non-edge, and
	\begin{equation}\label{eq:HAC}
H_{AC}(k)
	=
	\frac{
		\left(
		x_a\sqrt{x_c^2+y_c^2}
		\cos\left(\frac{2\pi}{m}-\frac{2\pi k}{m}\right)
		+ac-1
		\right)^2
	}{
		Q_AQ_C
	}.
\end{equation}
	For \(2\le k\le mp-2\), the \(B\)-\(B\) non-edge values are
	\begin{equation}\label{eq:HBB}
H_{BB}(k)
	=
	\frac{
		\left(
		(x_b^2+y_b^2)\cos\frac{2\pi k}{m}
		+
		z_b^2\cos\frac{2\pi k}{mp}
		+b^2-1
		\right)^2
	}{
		Q_B^2
	}.
\end{equation}
	For \(k\in\{0,\dots,mp-1\}\setminus\{0,1\}\), the \(B\)-\(C\) non-edge values are
\begin{equation}\label{eq:HBC}
\begin{aligned}
H_{BC}(k)=\frac1{Q_BQ_C}\Bigg(&
\sqrt{(x_b^2+y_b^2)(x_c^2+y_c^2)}
\cos\left(\frac{\pi}{m}-\frac{2\pi k}{m}\right)\\
&+z_b\sqrt{z_c^2+w_c^2}
\cos\left(\frac{\pi}{mp}-\frac{2\pi k}{mp}\right)
+bc-1\Bigg)^2.
\end{aligned}
\end{equation}
	For \(2\le k\le mp-2\), the \(C\)-\(C\) non-edge values are
	\begin{equation}\label{eq:HCC}
H_{CC}(k)
	=
	\frac{
		\left(
		(x_c^2+y_c^2)\cos\frac{2\pi k}{m}
		+
		(z_c^2+w_c^2)\cos\frac{2\pi k}{mp}
		+c^2-1
		\right)^2
	}{
		Q_C^2
	}.
\end{equation}
	Finally,
	\begin{equation}\label{eq:HAO}
H_{A O_c}
	=
	\frac{
		\left(a/c-1\right)^2
	}{
		Q_AQ_{O_c}
	},
\end{equation}
	and
	\begin{equation}\label{eq:HBO}
H_{B O_c}
	=
	\frac{
		\left(b/c-1\right)^2
	}{
		Q_BQ_{O_c}
	}.
\end{equation}

\begin{lemma}[Uniform separation estimate]\label{lem:uniform-separation}
For \(p\in\{2,3\}\) and \(m\geq5\), all the radicals in
\eqref{eq:R}--\eqref{eq:wc-value} are real and
\(Q_A,Q_B,Q_C,Q_{O_c}>0\).  Moreover, every non-edge pair has negative
Lorentzian inner product and satisfies \(H>1\).  The smallest
separation in the entire family occurs for \(p=3,m=5\) in the
\(A\)-\(C\) family, where
\[
 H_{AC}(1)=1.0149343201\ldots.
\]
\end{lemma}

\begin{proof}
The complete derivative reductions and rational endpoint enclosures are
recorded in the ancillary analytic note
\emph{Analytic verification of the uniform separation estimates}, supplied
with the article and archived in the project repository~\cite{HCPd}.
We give here the reductions needed to indicate the argument.

Retain \(t=2\pi/m\), and put
\[
q=\cos t,\qquad q_p=\cos\frac{t}{p}.
\]
Write \(\Lambda_p(t)\) for the expression \(\Lambda_p(m)\) in
\eqref{eq:Lambda} after this substitution.  We also abbreviate
\[
 U=1-\rho_pt^2,\qquad V=1-\sigma_pt^2,
 \qquad E_U=1-qU,\qquad E_V=1-qV.
\]
Thus \(0<t\leq2\pi/5\) and \(0<q<q_p<1\).  The two values of
\(\rho_p\) are smaller than \(1/4\), and all four parameters
\(\rho_p,\sigma_p\) are at most \(13/50\).  Since
\((13/50)(2\pi/5)^2<1\), the quantities \(U,V,E_U,E_V\) are positive.
Moreover, if
\[
 d_B=\frac{2q}{(1+q)U}=1-b^2,
\]
then
\[
 \frac{1-q}{1+q}=\tan^2\frac t2>\frac{t^2}{4}>\rho_pt^2,
\]
 and hence \(0<d_B<1\).  Differentiating \eqref{eq:Lambda}, clearing the
 positive radical factors, and using \(\sin x<x<\tan x\) on the relevant
 half-angles gives
 \[
  \Lambda_2'(t)>\frac{18}{25}t,
  \qquad
  \Lambda_3'(t)>\frac{6}{25}t
  \qquad (0<t\leq2\pi/5).
 \]
 Hence \(\Lambda_p\) is strictly increasing.  Direct expansion at the
 origin gives \(\lim_{t\to0}\Lambda_p(t)=1\), and therefore
\begin{equation}\label{eq:lambda-positive}
 \Lambda_p(t)>1.
\end{equation}
There is a useful interpretation of the definition of \(\mathcal M_p\).  Write
\[
 d_B^{-1/2}=\cosh r,
 \qquad \Lambda_p(t)=\cosh s
 \qquad (r,s>0).
\]
Then the addition formula for \(\cosh\) gives
\[
 \mathcal M_p(m)=\cosh(r+s)>1,
\]
and consequently all the radicals in the construction are real.  If
\(d_C=\mathcal M_p(m)^{-2}=1-c^2\), substitution in the Lorentzian norms gives the
particularly simple identities
\[
\begin{aligned}
 Q_A&=\frac{1-q}{q},\\
 Q_B&=d_B\left((1-q)U+\frac{1-q_p}{q_p}E_U\right),\\
 Q_C&=d_C\left((1-q)V+\frac{1-q_p}{q_p}E_V\right),\\
 Q_{O_c}&=\frac{1}{\mathcal M_p(m)^2-1}.
\end{aligned}
\]
All four quantities are therefore positive.

It remains to estimate the non-edges.  It is convenient to work before
squaring and to put
\[
 D_{XY}(k)=
 -\frac{\langle n(X_1),n(Y_{k+1})\rangle}{\sqrt{Q_XQ_Y}}.
\]
Thus \(H_{XY}(k)=D_{XY}(k)^2\), and it suffices to prove
\(D_{XY}(k)>1\).  For fixed \(t\), the denominators in
\eqref{eq:HAA}--\eqref{eq:HCC} are independent of \(k\).  Applying
\[
 \cos((k+1)\theta)-\cos(k\theta)
 =-2\sin\frac{(2k+1)\theta}{2}\sin\frac\theta2
\]
to the unsquared numerators, and writing \(k=jm+r\) with
\(0\leq r<m\), shows that their first differences have constant sign on
each half of the \(p\) arcs, with at most one change of sign.  Reflection
symmetry and comparison of the arc endpoints leave only \(k=2\) in the
\(AA\), \(AB\), and \(BC\) families, \(k=1\) in the \(AC\) family, and
\(k=m\) in the \(BB\) and \(CC\) families.  The two families involving
\(O_c\) have no angular parameter.

Let \(\Phi_{XY,p}(t)\) denote the corresponding one-variable function
after substituting the expressions for \(Q_A,Q_B,Q_C,Q_{O_c}\).
Differentiating and clearing positive radical factors gives the following
monotonicity properties.  For \(p=2\), the functions
\[
 \Phi_{AA,2},\ \Phi_{AB,2},\ \Phi_{AC,2},\ \Phi_{BC,2}
\]
are strictly decreasing, whereas \(\Phi_{AO_c,2}\) and
\(\Phi_{BO_c,2}\) are strictly increasing.  Each of
\(\Phi_{BB,2}\) and \(\Phi_{CC,2}\) has a unique critical point, which
is a strict maximum.  For \(p=3\), the functions
\[
 \Phi_{AA,3},\ \Phi_{AB,3},\ \Phi_{AC,3},\
 \Phi_{BC,3},\ \Phi_{AO_c,3}
\]
are strictly decreasing, while \(\Phi_{BB,3}\) and \(\Phi_{CC,3}\)
are strictly increasing; \(\Phi_{BO_c,3}\) has a unique critical point,
again a strict maximum.  These assertions follow directly from
\(\sin x<x<\tan x\) after the derivative numerators are simplified.

It follows that every minimum occurs at \(t=2\pi/5\), or is approached as
\(t\to0\).  The resulting endpoint calculation is summarized below.  An
entry of the form \(m\to\infty\) denotes a strict infimum, not attained
for finite \(m\).
\begingroup
\small
\[
\begin{array}{c@{\qquad}cc@{\qquad}cc}
\toprule
&\multicolumn{2}{c}{\text{minimizing regime}}
&\multicolumn{2}{c}{\text{lower bound for }D_{XY}^2}\\
\text{family}&p=2&p=3&p=2&p=3\\
\midrule
AA&(5,2)&(5,2)&2.618&2.618\\
AB&(5,2)&(5,2)&2.055&2.286\\
AC&(5,1)&(5,1)&1.506&1.014\\
BB&(5,5)&m\to\infty,\ k=m&1.559&1.020\\
BC&(5,2)&(5,2)&2.819&2.769\\
CC&(5,5)&m\to\infty,\ k=m&1.559&1.638\\
AO_c&m\to\infty&m=5&3.375&2.193\\
BO_c&m\to\infty&m\to\infty&1.500&1.494\\
\bottomrule
\end{array}
\]
\endgroup
\tablegap

As an illustration, the \(AA\) family reduces identically to
\[
 H_{AA}(k)=
 \left(\frac{q-\cos(kt)}{1-q}\right)^2,
\]
so its minimum occurs at \(k=2,m=5\), where
\[
 H_{AA}(2)=(1+2\cos(2\pi/5))^2
 =\frac{3+\sqrt5}{2}>2.618.
\]
The remaining entries are obtained in the same way by direct evaluation
of the endpoint expressions; the decimals in the table are downward
roundings.  No sampling in \(m\) is involved.  In the closest case,
direct interval evaluation gives
\[
 1.0149343<H_{AC}(1)<1.0149344
 \qquad (p=3,m=5).
\]
The first-difference calculation before squaring also gives
\(D_{XY}(k)>1\) in every family, so every corresponding Lorentzian inner
product is negative.  This proves all the assertions.
\end{proof}

\begin{remark}[Limitation of the three-layer cyclic ansatz]
\label{remark:three-layer-obstruction}
There is a simple obstruction to the standard three-layer ansatz of
Subsection~\ref{subsec:construction} for \(p\geq5\), and for \(p=4\) when
\(m\geq6\).  It is already visible in the \(A\)- and \(B\)-layers.  The
\(A_1A_2\), \(A_1B_1\), \(A_1B_2\), and
\(B_1B_2\) edge equations, together with the non-edge inequality for
\(B_1,B_{m+1}\), imply the necessary condition
\[
 2\cos\frac{2\pi}{mp}>
 \left(1+\cos\frac{2\pi}{p}\right)
 \left(1+2\cos\frac{2\pi}{m}\right).
\]
For \(p\geq5\), the right-hand side is larger than \(2\); for \(p=4\)
and \(m\geq6\), it is at least \(2\).  Both conclusions contradict the
strict upper bound \(2\cos(2\pi/(mp))<2\).

The remaining case \((p,m)=(4,5)\) is not detected by this argument.
Introducing
\[
 s=\frac{x_c^2+y_c^2}{1-c^2}.
\]
reduces the remaining \(C\)-layer and \(A\)-\(C\) conditions to an explicit
one-variable feasibility problem.  Numerical evaluation gives incompatible
upper and lower bounds for
\[
 \frac{1-bc}{\sqrt{(1-b^2)(1-c^2)}},
\]
and explains why the following collar construction is natural. 
\end{remark}

\subsection[The G453 collar]
{The \texorpdfstring{\(G_{4,5,3}\)}{G453} collar}
\label{subsection:collar}
Recall that the triangulation \(\mathcal T_{p,m}\) of the two-complex \(X_p\)
has the ordered layers
\[
 A-B-C-O_c.
\]
Here each letter other than \(O_c\) denotes a cyclic layer, consecutive
vertices within a layer are joined, and two consecutive cyclic layers are
triangulated by joining the \(i\)-th vertex of the outer layer to the
\(i\)-th and \((i+1)\)-st vertices of the inner layer.  In the \(A\)-layer
the labels are repeated with period \(m\), whereas the other cyclic layers
have period \(pm\).

For \(l\geq 0\), let \(\mathcal T_{p,m,l}\) be obtained by inserting \(l\)
additional cyclic layers between \(C\) and the cone vertex.  Thus
\[
\mathcal T_{p,m,0}=\mathcal T_{p,m}.
\]
For \((p,m,l)=(4,5,3)\), we denote the inserted layers by \(D,E,F\).
The order of all layers is therefore
\begin{equation}\label{eq:layer-order-453}
 A-B-C-D-E-F-O_c.
\end{equation}
The corresponding complex is shown in Figure~\ref{fig:triangulationadd3}.

\begin{lemma}\label{lem:collar-boundary}
The complex \(\mathcal T_{4,5,3}\) is a flag-no-square triangulation of
\(X_4\).  Consequently, \(G_{4,5,3}\) is Gromov hyperbolic and
\(\partial G_{4,5,3}\cong\Pi_4\).
\end{lemma}

\begin{proof}
Each inserted cyclic layer merely replaces the triangulated cone between
the preceding layer and \(O_c\) by one more triangulated annular strip and
a smaller cone, so the underlying PL space remains \(X_4\).  The same
local adjacency check as for Proposition~\ref{prop:triangulation} shows
that every clique is a face and every embedded four-cycle has a diagonal.
Thus the nerve is flag-no-square.  Moussong's criterion gives
hyperbolicity, and the boundary identification follows from the
degree-four form of Dranishnikov's argument described in
Remark~\ref{remark:defP}.
\end{proof}

\begin{figure}
    \centering
    \includegraphics[width=0.9\linewidth]{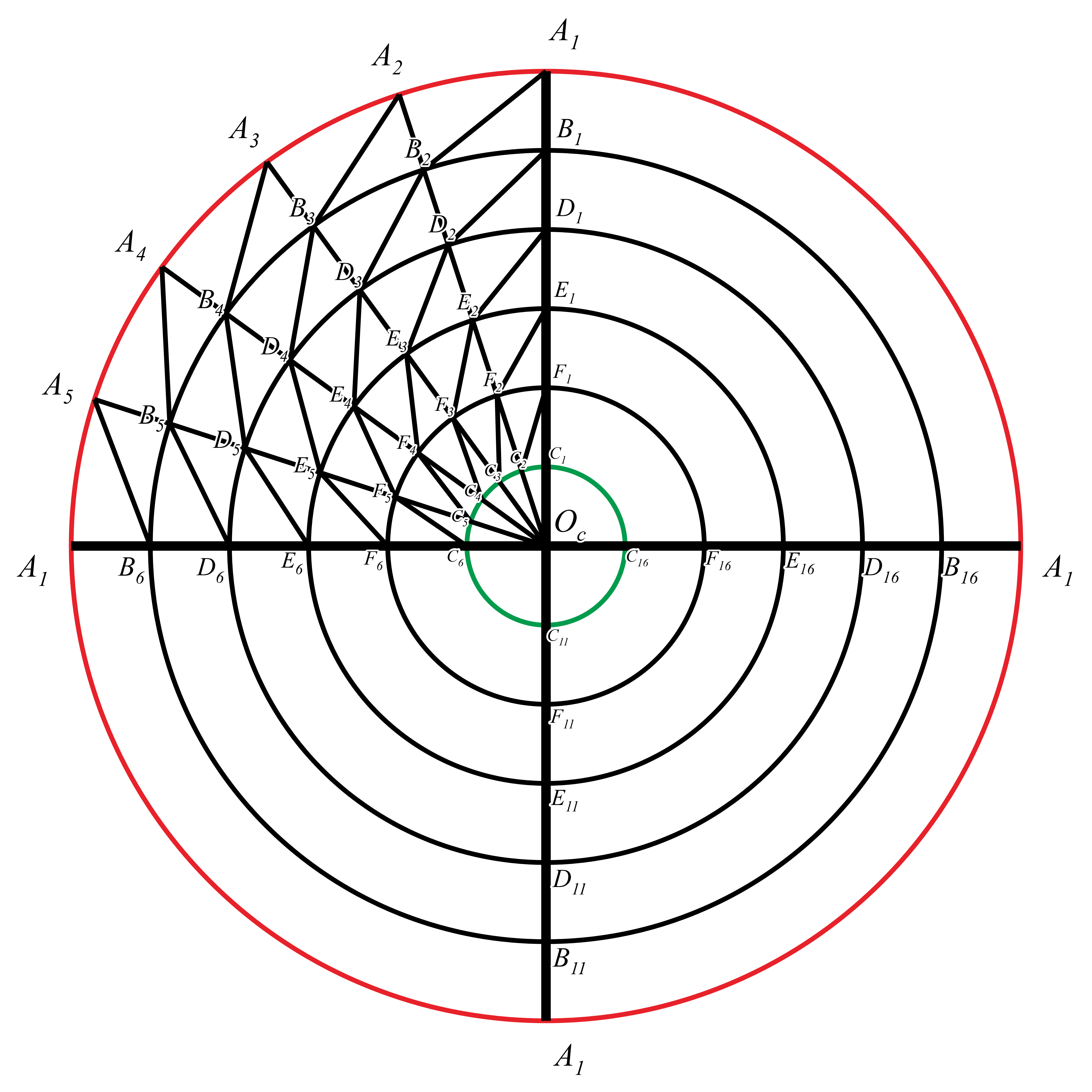}
    \caption{The flag-no-square triangulation
    \(\mathcal T_{4,5,3}\) of \(X_4\).}
    \label{fig:triangulationadd3}
\end{figure}

We denote the three inserted layers of $\mathcal{T}_{4,5,3}$ by \(D_i\),
\(E_i\), and \(F_i\), as shown in Figure~\ref{fig:triangulationadd3}.
Their initial pole vectors have the form
\[
\begin{aligned}
n(D_1)&=(x_d,y_d,z_d,w_d,d,1)^{\mathsf T},\\
n(E_1)&=(x_e,y_e,z_e,w_e,e,1)^{\mathsf T},\\
n(F_1)&=(x_f,y_f,z_f,w_f,f,1)^{\mathsf T}.
\end{aligned}
\]
For \(X_i\in\{D_i,E_i,F_i\}\) and \(1\leq i\leq20\), set
\[
 n(X_i)=S^{i-1}n(X_1).
\]
Since \(F\) is the innermost cyclic layer in
\eqref{eq:layer-order-453}, the pole of the cone reflection is
\begin{equation}\label{eq:Oc-pole-453}
 n(O_c)=(0,0,0,0,1/f,1)^{\mathsf T}.
\end{equation}
In particular, \(\langle n(F_i),n(O_c)\rangle=0\) for every \(i\).
The following decimal values approximate the certified solution described
in Proposition~\ref{prop:certified-453}; twenty digits are displayed.

\begin{center}
\begin{minipage}{0.88\textwidth}
\small
\begin{multicols}{2}
\noindent
\(a = 0\)

\(x_a = 1.79890743994786727226\)

\(b = 0\)

\(x_b = 0.55589297025142117199\)

\(y_b = 0.40387988390687641268\)

\(z_b = 0.94765814150418754522\)

\(c = 0.37871672186762506129\)

\(x_c = 0.06620923121333717465\)

\(y_c = 0.20377106091432812103\)

\(z_c = 0.92955010124102065292\)

\(w_c = 0.14722627253866626619\)

\(d = 0.63500241411989255536\)

\(x_d = -0.07722547088248804537\)

\(y_d = 0.15070963928862504633\)

\columnbreak

\(z_d = 0.75453785267906673625\)

\(w_d = 0.22099320594624878839\)

\(e = 0.81317731860166608815\)

\(x_e = -0.22245295031144918241\)

\(y_e = 0.15935754342818882822\)

\(z_e = 0.50540331391642384392\)

\(w_e = 0.27642821133104210787\)

\(f = 0.93028963143779908052\)

\(x_f = -0.17391442856664743322\)

\(y_f = -0.02416523077263075843\)

\(z_f = 0.31022447960072457765\)

\(w_f = 0.18769511896909947187\).
\end{multicols}
\end{minipage}
\end{center}
\tablegap
\begin{prop}[Certified Gram data]\label{prop:certified-453}
There is an exact collection of pole vectors in a neighborhood of the
displayed decimal data such that, for the nerve with ordered layers
\(A-B-C-D-E-F-O_c\), the following hold:
\begin{enumerate}
    \item every pole vector is spacelike;
    \item \(\langle n(X),n(Y)\rangle=0\) for every edge \(XY\) of the nerve;
    \item for every non-edge pair \(X,Y\),
    \[
       H(n(X),n(Y))>1.0627.
    \]
    \item the Gram matrix of the normalized pole vectors has non-positive
    off-diagonal entries and signature \((5,1,100)\).
\end{enumerate}
\end{prop}

\begin{proof}
The verification is computer-assisted.  It uses the directed real ball
arithmetic implemented in Arb~\cite{Johansson:2017}.

The complete certificate is the versioned ancillary Python file
\path{verify_g453_arb.py}, submitted with this
article and archived with the other supplementary material~\cite{HCPd}.
It was run with \texttt{python-flint}, version~0.8.0~\cite{pythonflint}.
We describe the certificate so that the numerical part of the proof is
reproducible.
We set \(a=b=0\) and fix the eight parameters
\[
 x_c,c,x_d,d,x_e,e,x_f,f
\]
to be the terminating decimals displayed above, regarded as exact rational
numbers.  The remaining sixteen variables are
\[
 x_a,x_b,y_b,z_b,
 y_c,z_c,w_c,y_d,z_d,w_d,y_e,z_e,w_e,y_f,z_f,w_f.
\]
They are determined by sixteen edge equations: the equation
\[
 \langle n(A_1),n(A_2)\rangle=0;
\]
two equations for each of the five consecutive pairs of cyclic layers
\(A\)-\(B\), \(B\)-\(C\), \(C\)-\(D\), \(D\)-\(E\), and \(E\)-\(F\);
and one equation between consecutive poles in each of the layers
\(B,C,D,E,F\).  The \(F_iO_c\) equations are automatic from
\eqref{eq:Oc-pole-453}.

The script works with \(256\)-bit Arb balls.  Let
\(\mathbf x_0\) be the 80-digit center recorded in the script and let
\[
 \mathbf X=\mathbf x_0+[-10^{-40},10^{-40}]^{16}.
\]
Writing \(\mathcal F:\mathbb R^{16}\to\mathbb R^{16}\) for the edge-equation
map and \(C\) for a numerical inverse of \(D\mathcal F(\mathbf x_0)\), the
script verifies, with outward rounding, the strict Krawczyk
inclusion~\cite{Krawczyk:1969}
\[
 \mathbf x_0-C\mathcal F(\mathbf x_0)
 +\bigl(I-C D\mathcal F(\mathbf X)\bigr)(\mathbf X-\mathbf x_0)
 \subset \operatorname{int}(\mathbf X).
\]
Consequently, \(\mathcal F\) has a unique zero in \(\mathbf X\).

At this zero, interval evaluation gives the following conservative lower
bounds for the Lorentzian norms:
\begingroup
\small
\[
\begin{array}{c@{\qquad}ccccccc}
\toprule
X&A&B&C&D&E&F&O_c\\
\midrule
\langle n(X),n(X)\rangle
&2.2360&0.3701&0.0750&0.0500&0.0679&0.0277&0.1554\\
\bottomrule
\end{array}
\]
\endgroup
\tablegap
The script then enumerates all pairs among the \(106\) poles, using the
combinatorics in \eqref{eq:layer-order-453}.  After the prescribed edge pairs
are removed, it checks all \(5240\) non-edge pairs.  Direct ball evaluation
gives
\[
 \min_{XY\text{ a non-edge}}H(n(X),n(Y))
 >1.0627365391573188>1.0627.
\]
The same ball evaluation certifies that every non-edge inner product is
strictly negative.  Since the edge inner products vanish, the normalized
Gram matrix therefore has non-positive off-diagonal entries.

It remains to determine its inertia.  Let \(N\) be the \(6\times106\)
matrix whose columns are the (unnormalized) pole vectors.  The interval
enclosure for the minor with columns
\(A_1,B_1,C_1,D_1,E_1,F_1\) is
\[
[-0.001570436952200370241728775925
 \mathbin{\pm}3.10\cdot10^{-40}],
\]
which does not contain zero.  Hence \(\operatorname{rank}N=6\).  Column
normalization does not change the rank, so if \(U\) is the matrix of
normalized pole vectors,
then
\[
  G=U^{\mathsf T}L U
\]
has, by Sylvester's law of inertia, five positive eigenvalues, one negative
eigenvalue, and \(106-6=100\) zero eigenvalues.  Thus its signature is
\((5,1,100)\).

For reference, a high-precision numerical diagonalization gives the six
nonzero eigenvalues
\begingroup
\small
\[
\begin{array}{r@{\qquad}r@{\qquad}r}
-907.23759098,&50.34372369,&50.34372369,\\
188.72171332,&361.91421514,&361.91421514.
\end{array}
\]
\endgroup
\tablegap
These decimals are included only as a numerical check; the rigorous
signature assertion follows from the nonvanishing minor and Sylvester's
law.
\end{proof}

\subsection{Asymmetric right-angled nerves}
\label{subsection:asysub}

The obstruction in Remark~\ref{remark:three-layer-obstruction} is a
constraint on the cyclic ansatz, rather than on right-angled
realizations themselves.  We therefore keep all Coxeter angles equal to
\(\pi/2\) and replace the symmetric nerve by a suitable asymmetric one.

For \(2\leq p\leq 8\), let \(\mathcal K_p\) be the triangulation shown in
Figure~\ref{fig:asymmetric-nerves}.  Each panel represents a triangulated
disk whose equally labelled boundary vertices are identified.  The
induced attaching map has degree \(p\), and hence the quotient is
\(X_p\).  Every bounded triangular region is a two-simplex, so the
figure records the full simplicial structure, not only the
one-skeleton.

\begin{figure}[t]
  \centering
  \includegraphics[width=1.1\textwidth]{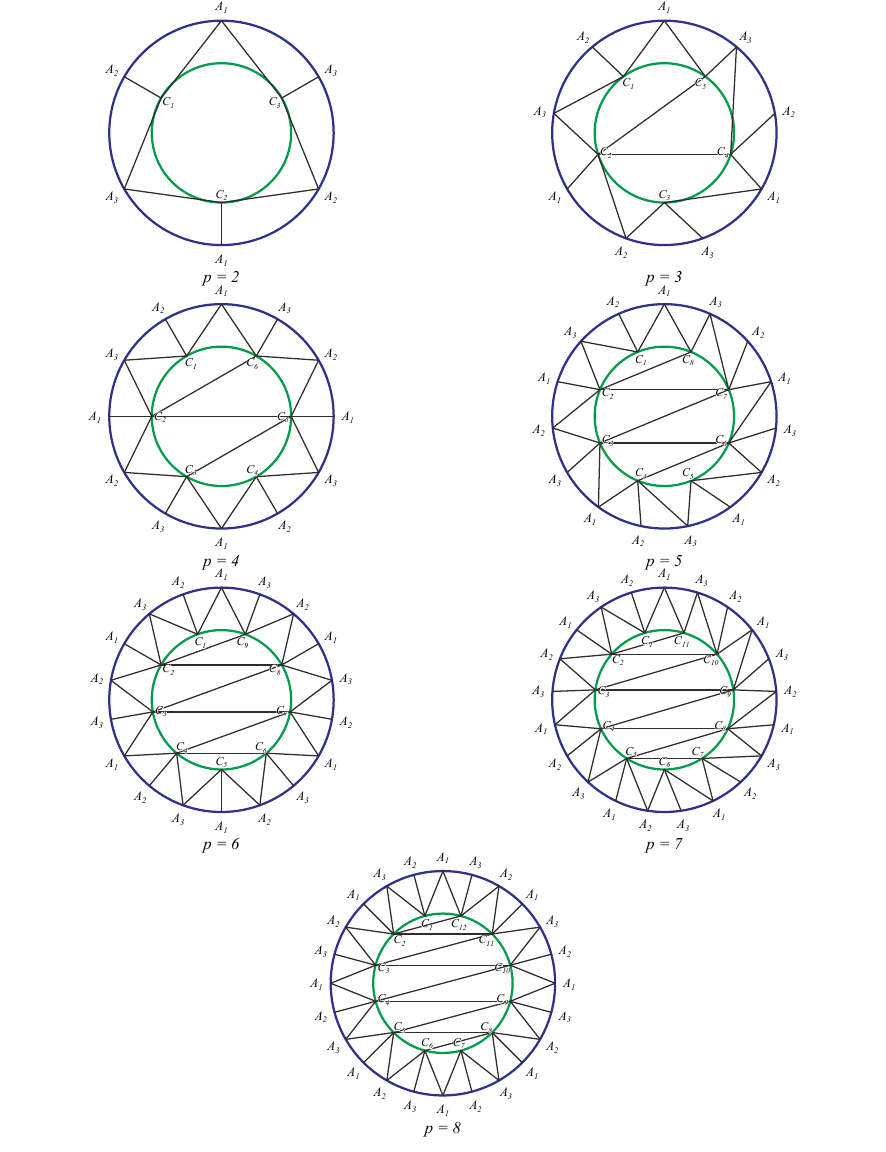}
  \caption{Triangulations \(\mathcal K_p\) of \(X_p\), for
  \(2\leq p\leq 8\).  Boundary vertices with the same label are
  identified.}
  \label{fig:asymmetric-nerves}
\end{figure}

Put \(c_p=\lceil3p/2\rceil\).  For \(2\leq p\leq8\), the face vectors
of these triangulations are
\[
 (f_0,f_1,f_2)
 =\bigl(c_p+3,\,3p+3c_p,\,3p+2c_p-2\bigr).
\]
In particular,
\(f_0(\mathcal K_p)=\lceil3p/2\rceil+3\), in agreement with
\cite[\S2.1]{LofanoLutz:2021}.

Applying Dranishnikov's special subdivision, as illustrated in
Figure~\ref{fig:Dsubdivision}, to every two-simplex of \(\mathcal K_p\) gives a
complex \(\mathcal K_p^{\#}\) with the same underlying space as
\(\mathcal K_p\).  By \cite[Proposition~2.1]{Dran:1999}, the resulting
complex is flag-no-square.  For an edge \(\{X,Y\}\) of \(\mathcal K_p\),
denote the new edge vertex by
\(e_{XY}=e_{YX}\).  For a two-simplex \(\{X,Y,Z\}\), denote its three new
face vertices by
\(
 f_{X;YZ},f_{Y;XZ},f_{Z;XY},
\)
where \(f_{X;YZ}=f_{X;ZY}\); the entry before the semicolon records the
distinguished original vertex.

\begin{figure}[H]
  \centering
  \includegraphics[width=0.7\linewidth]{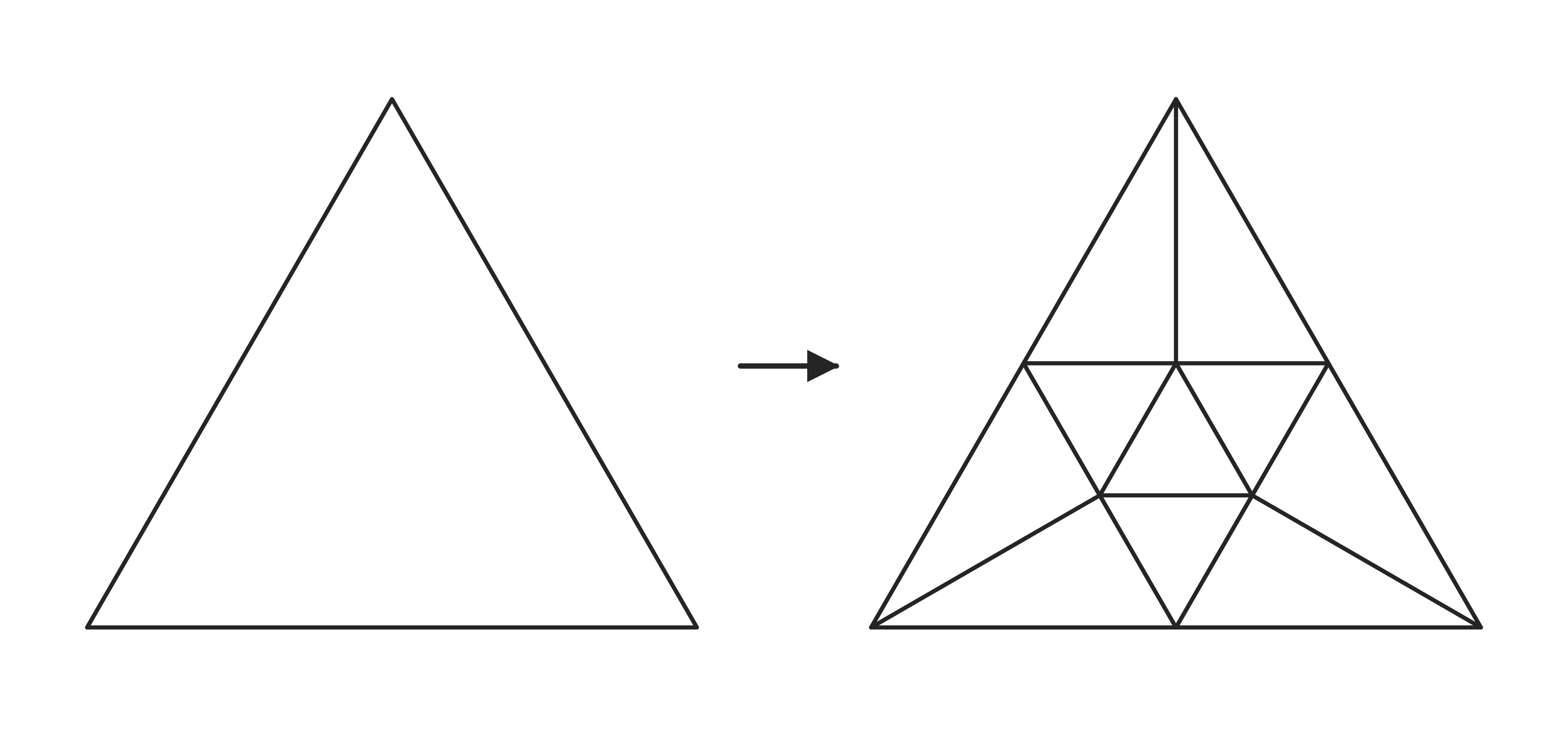}
  \caption{The local subdivision rule for a two-simplex.}
  \label{fig:Dsubdivision}
\end{figure}

For a simplicial complex \(Y\), write \(Y^{(1)}\) for its one-skeleton.
If \(\mathcal G\) is a finite graph, write \(V(\mathcal G)\) and
\(E(\mathcal G)\) for its vertex and edge sets, respectively, and let
\(W_{\mathcal G}\) denote the right-angled Coxeter group with defining
graph \(\mathcal G\).  Set
\[
 \mathcal G_p=(\mathcal K_p^{\#})^{(1)},
 \qquad
 W_p=W_{\mathcal G_p}.
\]
Since \(\mathcal K_p^{\#}\) is flag, it is precisely the nerve of
\(W_p\).

\begin{table}[H]
\centering
\small
\begin{tabular}{c*{7}{r}}
\toprule
\(p\) & 2 & 3 & 4 & 5 & 6 & 7 & 8\\
\midrule
\(\lvert V(\mathcal G_p)\rvert\)
 & 51 & 83 & 105 & 137 & 159 & 191 & 213\\
\(\lvert E(\mathcal G_p)\rvert\)
 & 150 & 252 & 324 & 426 & 498 & 600 & 672\\
\bottomrule
\end{tabular}
\caption{Numbers of vertices and edges in the defining graphs
\(\mathcal G_p\).}
\label{tab:asymmetric-graph-sizes}
\end{table}
\tablegap

\begin{prop}[Certified asymmetric pole data]
\label{prop:asymmetric-certificate}
For every \(p\in\{2,3,4,5,6,7,8\}\), there is a collection of spacelike
vectors
\[
 \{n_v\in\mathbb R^{5,1}:v\in V(\mathcal G_p)\}
\]
oriented so that
\[
 \langle O,n_v\rangle<0
 \qquad\text{for every }v\in V(\mathcal G_p),
\]
where \(O=(0,0,0,0,0,1)^{\mathsf T}\in\mathbf H^5\),
and satisfying
\[
 \langle n_v,n_w\rangle=
 \begin{cases}
  1,&v=w,\\
  0,&\{v,w\}\in E(\mathcal G_p),\\
  <-1,&\{v,w\}\notin E(\mathcal G_p).
 \end{cases}
\]
\end{prop}

\begin{proof}
Because no symmetry is imposed, the parameter space for the pole
configurations is considerably larger than in the cyclic constructions
above.  We therefore use the numerical-search and rigorous-certification
framework developed in Section~\ref{sec:certification-framework}.
Briefly, the approximate pole matrices and the corresponding
ball-arithmetic certificates are contained in the ancillary material
\cite{HCPd,RACGRealizer}; the complete approximate Lorentz-unit matrix
for \(p=5\) is displayed in
Appendix~\ref{appendix:asymmetric-data}.  For each \(p\),
Theorem~\ref{thm:criterion} certifies an exact zero of the diagonal and
edge equations in a neighborhood of the stored approximate matrix.
Interval evaluation throughout the same neighborhood shows that every
non-edge inner product is strictly less than \(-1\).  The same interval
bounds verify \(\langle O,n_v\rangle<0\) for every \(v\).
\end{proof}

%The preceding experiments suggest two possible directions for further constructions.  One is to abandon highly symmetric triangulations in favor of asymmetric ones, perhaps after applying a Przytycki--\'Swi\k{a}tkowski subdivision as in \cite{DLMR:2025}. Another is to proceed inductively from a small subdivision, adding pole vectors individually or in collections arising as facet normals of hyperbolic Coxeter polytopes of dimension at most five, chosen to be as symmetric as possible,while testing the required inequalities at each stage.

\subsection{Discreteness and faithfulness}\label{subsec:discrete}
We give a single geometric argument for the representations constructed
above.  It applies to \(G_{2,m}\), \(G_{3,m}\), \(G_{4,5,3}\), and
\(W_p\), and will also be used for the acute-angled realizations in
Section~4.

Let \(\mathcal S\) denote the relevant Coxeter generating set.  For
\(X\in\mathcal S\), write \(n_X\) for its spacelike pole,
\[
 q_X=\langle n_X,n_X\rangle,\qquad
 u_X=\frac{n_X}{\sqrt{q_X}},
\]
and let \(B_X=u_X^\perp\cap\mathbf H^5\) be the corresponding mirror.
The poles in each construction are oriented so that
\(\langle O,n_X\rangle<0\).  Hence, after a positive rescaling, which does
not change the corresponding reflection, every pole may be assumed to
have last coordinate \(1\).  It follows that
\(\langle O,u_X\rangle=-q_X^{-1/2}<0\).  We choose the half-space
\[
 \mathcal H_X=\{x\in\mathbf H^5:\langle x,u_X\rangle\leq0\}
\]
and define
\[
 \mathcal D=\bigcap_{X\in\mathcal S}\mathcal H_X.
\]
Thus \(O\) lies in the interior of \(\mathcal D\).

\begin{lemma}\label{lemma:intersection}
For distinct \(X,Y\in\mathcal S\), the mirrors \(B_X\) and \(B_Y\)
intersect if and only if \(\{X,Y\}\) is an edge of the Coxeter nerve.  If
that edge has weight \(k\), their dihedral angle is \(\pi/k\); if there is
no edge, the two mirrors are hyperparallel.

\end{lemma}

\begin{proof}
For an edge of weight \(k\), the construction gives
\[
 \langle u_X,u_Y\rangle=-\cos(\pi/k),
\]
so the mirrors meet at angle \(\pi/k\).  For a non-edge, the separation
estimates give
\[
 \langle u_X,u_Y\rangle<0,
 \qquad
 \langle u_X,u_Y\rangle^2=H(n_X,n_Y)>1.
\]
Thus \(\langle u_X,u_Y\rangle<-1\), which is precisely the
hyperparallel case.
\end{proof}

\begin{lemma}[Realization of the Coxeter faces]\label{lemma:connected}
Let \(T\subset\mathcal S\) be a simplex of the Coxeter nerve.  Then
\[
 \mathcal D\cap\bigcap_{X\in T}B_X
\]
contains a point that lies on no mirror \(B_Y\) with \(Y\notin T\).
Consequently, every generator determines a genuine facet of
\(\mathcal D\), and every edge of the nerve determines a genuine ridge.

\end{lemma}

\begin{proof}
Let \(G_T=(\langle u_X,u_Y\rangle)_{X,Y\in T}\).  Since \(T\) is
spherical, \(G_T\) is positive definite.  Its diagonal entries are \(1\)
and its off-diagonal entries are non-positive.  Thus \(G_T^{-1}\) is
entrywise nonnegative: indeed, writing \(G_T=I-C_T\) with \(C_T\geq0\),
positive definiteness gives \(\rho(C_T)<1\), and hence
\[
 G_T^{-1}=\sum_{j=0}^{\infty}C_T^j\geq0.
\]
Set
\[
 \alpha=(q_X^{-1/2})_{X\in T},\qquad
 c=G_T^{-1}\alpha,
 \qquad
 v_T=O+\sum_{X\in T}c_Xu_X.
\]
Then \(c_X\geq0\), and for every \(X\in T\),
\[
 \langle v_T,u_X\rangle=-\alpha_X+(G_Tc)_X=0.
\]
If \(Y\notin T\), global non-positivity of the normalized Gram matrix gives
\[
 \langle v_T,u_Y\rangle
 =-q_Y^{-1/2}+\sum_{X\in T}c_X\langle u_X,u_Y\rangle
 \leq-q_Y^{-1/2}<0.
\]
Finally,
\[
\langle v_T,v_T\rangle
 =-1-2\alpha^{\mathsf T}c+c^{\mathsf T}G_Tc
 =-1-\alpha^{\mathsf T}c<0.
\]
Moreover, the last coordinate of \(v_T\) is
\[
 1+\sum_{X\in T}c_Xq_X^{-1/2}>0.
\]
After positive
normalization, \(v_T\) therefore represents the required point of
\(\mathbf H^5\).
\end{proof}

\begin{proof}[Proof of Proposition~\ref{prop:smallp}]
Proposition~\ref{prop:group} establishes the hyperbolicity and boundary
statements for \(G_{p,m}\), and Lemma~\ref{lem:collar-boundary} establishes
their counterparts for the collar group.  It remains to analyze the
reflection representations.

For \(p=2,3\), use the pole configurations of
Proposition~\ref{prop:uniform-smallp}; for \(G_{4,5,3}\), use the certified
configuration of Proposition~\ref{prop:certified-453}.  In both cases,
Lemmas~\ref{lemma:intersection} and~\ref{lemma:connected} show that
\(\mathcal D\) is a Coxeter polytope with precisely the prescribed facets
and ridges, the prescribed Coxeter angles, and non-empty interior.

Poincar\'e's polyhedron theorem \cite{Poincare:1983} (see also
Theorem~11.2 of \cite{LeeMarquis:2024}) identifies \(\mathcal D\) as a
fundamental polytope for the group generated by its facet reflections.
The resulting homomorphism from the abstract Coxeter group is therefore
faithful, and its image is discrete in \(\mathbf{PO}(5,1)\).

Finally, Lemma~\ref{lemma:intersection} shows that any two non-intersecting
facets of \(\mathcal D\) are hyperparallel, so \(\mathcal D\) has no
asymptotic pair of facets.  Theorem~4.12 of \cite{DHaglund:2013} therefore
implies that the reflection group is convex cocompact.
\end{proof}

\begin{proof}[Proof of Proposition~\ref{prop:morep}]
The complex \(\mathcal K_p^{\#}\)
is flag-no-square and has underlying space \(X_p\); hence Moussong's
criterion implies that
\(W_p\) is word-hyperbolic, while the boundary theorem of
\cite{Dran:1997,Dran:1999}, in the form explained in
Remark~\ref{remark:defP}, identifies \(\partial W_p\) with \(\Pi_p\).

For the geometric part, use the certified pole configuration of
Proposition~\ref{prop:asymmetric-certificate}.  Its orientation places
\(O\) in the interior of \(\mathcal D\), and
Lemmas~\ref{lemma:intersection} and~\ref{lemma:connected} identify
\(\mathcal D\) as a Coxeter polytope with precisely the prescribed facets,
ridges, and dihedral angles.  Poincar\'e's polyhedron theorem then gives a
discrete and faithful reflection representation of \(W_p\).  Finally, the
strict non-edge inequalities make every pair of non-intersecting facets
hyperparallel.  Thus \(\mathcal D\) has no asymptotic pair of facets, and
\cite[Theorem~4.12]{DHaglund:2013} implies convex cocompactness.
\end{proof}

\begin{proof}[Proof of Theorem~\ref{thm:pontryaginsmallp}]
For \(p\in\{2,3,4,5,6,7,8\}\), take \(G_p=W_p\) and use the representation
provided by Proposition~\ref{prop:morep}.  Convex cocompactness identifies
the limit set equivariantly with the Gromov boundary of \(W_p\), which is
homeomorphic to \(\Pi_p\).
\end{proof}

\section{Proof of Theorem~\ref{thm:pontryaginlargep}}
\label{subsec:representationlargep}

Given the triangulation \(\mathcal C_{p,m}\) of \(X_p\) shown in
Figure~\ref{figure:triangulation2level}, we associate a Coxeter group
\(J_{p,m}\) as in the right-angled case, except that some edge weights
are different from \(2\).  In the order
\[
 (A_1A_2,\ A_1C_1,\ A_1C_2,\ C_1C_2,\ C_1O_c),
\]
the weights are
\[
 (2,3,3,2,3).
\]
All other labels are obtained by the \(\mathbb Z_{pm}\)-symmetry.  The
semidirect product \(J_{p,m}\rtimes\mathbb Z_{pm}\) has the following
symmetric presentation, where \(S\) denotes the cyclic symmetry used in
Subsection~\ref{subsec:construction}:

\[\left\langle A_{1}, C_{1}, O_c, S\Bigg| \begin{array}  {c} A_1^2=C_1^2=O_c^2=S^{mp}=\text{{\rm id}},\\ [6pt]
   (A_1SA_1S^{-1})^2=(C_1SC_1S^{-1})^2=(A_1 C_1)^3\\
   \qquad{}=(A_1S C_1 S^{-1})^3=(C_1O_c)^3=\text{{\rm id}},\\[6pt]
   S^m A_1 S^{-m}=A_1,~S^{-1} O_c S=O_c
	\end{array}\right\rangle.
\]

\begin{prop}\label{prop:groupnonrrightangled}
For \(m\geq5\) and \(p\geq2\), \(J_{p,m}\) is a hyperbolic Coxeter
group whose weighted nerve has underlying complex \(X_p\).
\end{prop}

\begin{proof}
The layer identifications are the same degree-\(p\) identifications used
for \(\mathcal T_{p,m}\), so the underlying complex is \(X_p\).  Every
triangle in the weighted nerve has Coxeter type \((2,3,3)\), hence is
spherical, and there is no clique of larger rank.  Thus no special
subgroup of rank at least three is affine.  Moreover, a product of two
infinite special subgroups would produce an induced four-cycle whose
 cross-edges all have weight \(2\); the layer-by-layer no-square check
excludes such a cycle.  Moussong's hyperbolicity criterion now implies
that \(J_{p,m}\) is word-hyperbolic.
\end{proof}

\begin{remark}\label{remark:nonrightangledboundary}

Theorem~B of \'{S}wi\k{a}tkowski \cite{Swiatkowski:2009} identifies the
boundary of a hyperbolic Coxeter group whose nerve is a closed
non-orientable PL manifold \(M\) with the Jakobsche space associated
with \(M\); the theorem does not require the Coxeter group to be
right-angled.  Since \(X_2\cong\mathbb{RP}^2\), it follows that
\(\partial J_{2,m}\cong\Pi_2\). For \(p>2\), however, \(X_p\) is not a manifold, so that theorem does
not apply.  Moreover, Dranishnikov's argument for
Proposition~\ref{prop:group} uses the right-angled structure.  We
conjecture that \(\partial J_{p,m}\cong\Pi_p\) for \(p=3,4\).

\end{remark}

\begin{figure}
	\centering
    \includegraphics[width=0.8\linewidth]{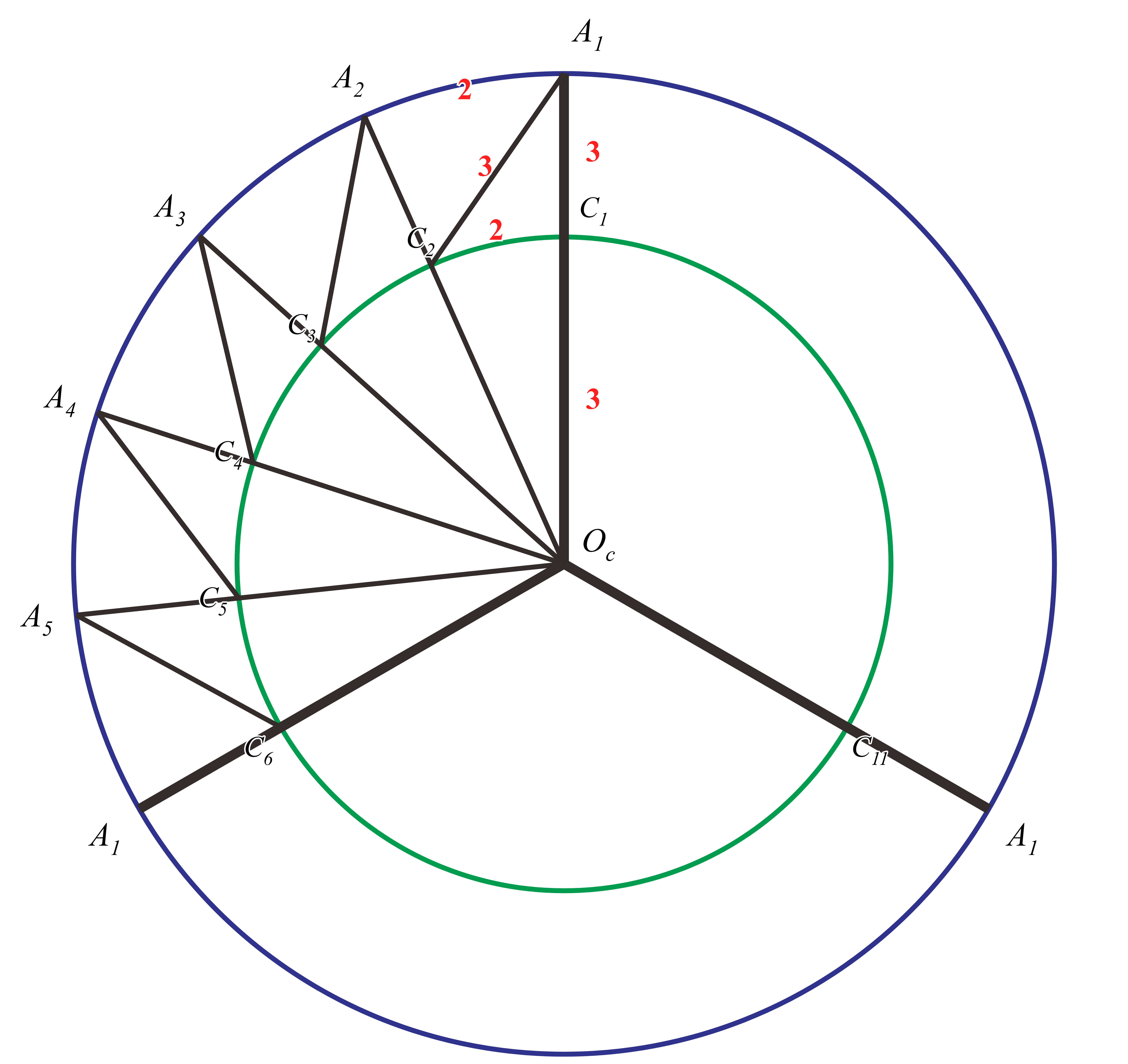}
	\caption{The weighted nerve \(\mathcal C_{3,6}\) of \(X_3\).}
\label{figure:triangulation2level}
\end{figure}

For a candidate collection of pole vectors, we must verify three types
of conditions:
\begin{enumerate}
    \item every pole vector \(n(V)\) is spacelike;
    \item if \(V,W\) are joined by an edge of weight \(k\), then
    \(H(n(V),n(W))=\cos^2(\pi/k)\);
    \item if \(V,W\) are not joined by an edge, then
    \(H(n(V),n(W))>1\).
\end{enumerate}
Here \(H\) is the scale-invariant function defined in
\eqref{eq:Hfunction}.

Notably, numerical optimization can locate possible zeros, but it proves neither the
existence of an exact zero nor the strict non-edge inequalities.  We
therefore audit each displayed candidate in two stages.  First, after fixing
the printed decimal values of \(a,c,w_c\), we apply the Krawczyk operator to
the five edge equations in the variables \(x_a,x_c,y_c,z_c,o_c\).  Here the
central pole is
\[
 n(O_c)=(0,0,0,0,o_c,1)^{\mathsf T};
\]
the coordinate \(o_c\) must be independent of \(c\), since the earlier
choice \(o_c=1/c\) would force \(H(n(C_1),n(O_c))=0\), corresponding to
weight \(2\) rather than weight \(3\).  Second, on
the resulting root enclosure we evaluate every non-edge using \(256\)-bit
Arb ball arithmetic.  The complete machine-checkable audit is contained in
the versioned ancillary Python file
\path{verify_nonright_arb.py}, submitted with this article and archived in
the supplementary repository~\cite{HCPd}.  It requires
\texttt{python-flint}, version~0.8.0~\cite{pythonflint}, and uses Arb ball
arithmetic~\cite{Johansson:2017}.  Running the file without arguments certifies
all three cases and exits successfully.

Throughout this audit \(m=5\): the rotation angles in the two Euclidean
coordinate planes are \(2\pi/5\) and \(2\pi/(5p)\), respectively.  Thus the
enumeration contains \(5\) poles of type \(A\), \(5p\) poles of type \(C\),
and the pole \(O_c\).  Figure~\ref{figure:triangulation2level} is merely a
schematic illustration of the general combinatorics.
The following values, with \(m=5\), are numerical candidates for the weight
system \((2,3,3,2,3)\).  The additional coordinate \(o_c\) is determined by
the \(C_1O_c\) angle equation.

\begingroup
\small
\[
\begin{array}{@{}lll@{}}
\begin{alignedat}{2}
p   &{}=2,\\
a   &{}=-0.6268817510226523,\\
x_a &{}=-1.401554716068,\\
c   &{}=0.5487408449899542,\\
x_c &{}=-0.647621030774,\\
y_c &{}=-0.470524220888,\\
z_c &{}=0.786829948613,\\
w_c &{}=0,\\
o_c &{}=1.278590646120
\end{alignedat}
&
\begin{alignedat}{2}
p   &{}=3,\\
a   &{}=0.48414887858672406,\\
x_a &{}=-1.574018170609,\\
c   &{}=-0.60870282888116,\\
x_c &{}=-0.565400750572,\\
y_c &{}=-0.410787690657,\\
z_c &{}=0.723766203244,\\
w_c &{}=0,\\
o_c &{}=-1.256338335087
\end{alignedat}
&
\begin{alignedat}{2}
p   &{}=4,\\
a   &{}=0.574393980597,\\
x_a &{}=1.472547825934,\\
c   &{}=-0.354053027688,\\
x_c &{}=0.563516558953,\\
y_c &{}=0.409418745315,\\
z_c &{}=0.812110133226,\\
w_c &{}=0.320143790304,\\
o_c &{}=-1.670750761584
\end{alignedat}
\end{array}
\]
\endgroup
\tablegap
\begin{prop}[Certified realization of the displayed candidates]
\label{prop:nonrightcertificate}
For each \(p\in\{2,3,4\}\), after \(a,c,w_c\) are fixed at their displayed
decimal values, the ancillary computation refines the remaining five
coordinates to a center \(x_0=(x_a,x_c,y_c,z_c,o_c)\).  There is a unique
exact solution of the five independent edge equations in the
radius-\(10^{-40}\) box about \(x_0\).  These roots realize the five
prescribed edge weights \((2,3,3,2,3)\).

\noindent For \(p=2\), all \(75\) non-edges satisfy
\[
 H>1+0.45.
\]
For \(p=3\), all \(145\) non-edges satisfy
\[
 H>1+0.10.
\]
For \(p=4\), all \(240\) non-edges satisfy
\[
 H>1+0.09.
\]
More precisely, the respective minimum intervals are
\begingroup
\small
\[
\begin{array}{c@{\qquad}c@{\qquad}c}
\toprule
p&\text{non-edge attaining the minimum}&\text{certified interval}\\
\midrule
2&C_1C_6&
[1.4569476819993073,\,1.4569476819993074]\\
3&C_1C_6&
[1.1082946366272705,\,1.1082946366272706]\\
4&C_1C_6&
[1.0929469998988836,\,1.0929469998988838]\\
\bottomrule
\end{array}
\]
\endgroup
\tablegap
In addition, for each \(p\), the Gram matrix of the normalized pole vectors
has non-positive off-diagonal entries and signature \((5,1,5p)\).
\end{prop}

\begin{proof}
Let \(F_p\colon\mathbb R^5\to\mathbb R^5\) consist of the
\(A_1A_2\) orthogonality equation, the squared \(A_1C_1\) and \(A_1C_2\)
angle equations, the \(C_1C_2\) orthogonality equation, and the squared
\(C_1O_c\) angle equation.  With
\(x_0\) equal to the refined decimal center and \(X=x_0+[-10^{-40},
10^{-40}]^5\), the ancillary computation verifies
\[
x_0-CF_p(x_0)+(I-CDF_p(X))(X-x_0)\subset\operatorname{int}(X),
\]
where \(C\) is a high-precision numerical approximation to
\(DF_p(x_0)^{-1}\).  The Krawczyk theorem~\cite{Krawczyk:1969} therefore gives a unique zero
in \(X\).  Arb evaluation on \(X\) verifies spacelikeness and the
required signs of the angle equations, and then enumerates all non-edges.
The rigorous lower bounds stated above prove the required strict separation
with \(\delta=0.45,0.10,0.09\), respectively.

The same interval evaluation verifies the sign before squaring: every
non-edge Lorentzian inner product is strictly negative.  On an edge the
normalized inner product is \(-\cos(\pi/k)\), where \(k\) is its weight.
Thus every off-diagonal entry of the normalized Gram matrix is
non-positive.

For the signature, let \(N_p\) be the \(6\times(5p+6)\) matrix of
unnormalized pole vectors.  The \(6\times6\) minor with columns
\(A_1,A_2,C_1,C_2,C_3,O_c\) is enclosed by
\begingroup
\small
\[
\begin{array}{c@{\qquad}c}
\toprule
p&\text{certified determinant interval}\\
\midrule
2&[-1.104069531291515701202126173267
   \mathbin{\pm}6.07\cdot10^{-39}]\\
3&[0.636027725134980401929890086905
   \mathbin{\pm}3.09\cdot10^{-39}]\\
4&[0.881146830017736918069741440403
   \mathbin{\pm}7.31\cdot10^{-39}]\\
\bottomrule
\end{array}
\]
\endgroup
\tablegap
None of these intervals contains zero, so \(\operatorname{rank}N_p=6\).
If \(U_p\) is obtained by normalizing the columns, then
\[
 Q_p=U_p^{\mathsf T}L U_p.
\]
Sylvester's law of inertia therefore gives five positive eigenvalues, one
negative eigenvalue, and \((5p+6)-6=5p\) zero eigenvalues.

For comparison, the six nonzero eigenvalues obtained by high-precision
numerical diagonalization are
\begingroup
\small
\[
\begin{array}{c@{\qquad}rrr}
\toprule
p&\multicolumn{3}{c}{\text{nonzero eigenvalues}}\\
\midrule
2&-20.02437674& 5.51760220& 5.51760220\\
 &  6.33102390& 9.32907423& 9.32907423\\[2pt]
3&-31.61063369& 5.70717524&10.26377867\\
 & 10.26377867&13.18795055&13.18795055\\[2pt]
4&-52.31598203& 4.12446136&16.64134978\\
 & 16.64134978&20.45441055&20.45441055\\
\bottomrule
\end{array}
\]
\endgroup
\tablegap
As above, these decimal eigenvalues are reported only for transparency;
the signature \((5,1,5p)\) is certified by the nonzero minor.
\end{proof}

\begin{proof}[Proof of Theorem~\ref{thm:pontryaginlargep}]
For \(p\in\{2,3,4\}\), set \(J_p=J_{p,5}\).
Proposition~\ref{prop:groupnonrrightangled} supplies the abstract hyperbolic
Coxeter groups.  Proposition~\ref{prop:nonrightcertificate} certifies
spacelike poles, all prescribed dihedral angles, and strict
hyperparallelism for every non-edge.
The mirror-intersection and face-realization arguments of
Lemmas~\ref{lemma:intersection} and~\ref{lemma:connected} therefore apply
to the resulting acute-angled Coxeter polytope.  Poincar\'e's theorem gives
a discrete, faithful reflection representation.  Convex cocompactness then
follows as in the proof of Theorem~\ref{thm:pontryaginsmallp}.
\end{proof}

\section{Numerical search and rigorous certification}
\label{sec:certification-framework}
The finite constructions in
Propositions~\ref{prop:certified-453},
\ref{prop:nonrightcertificate}, and
\ref{prop:asymmetric-certificate} have different combinatorial origins,
but they share a common computational structure: a floating-point search
produces an approximate pole configuration, after which interval
arithmetic certifies a nearby exact solution and all strict separation
inequalities.  This section formulates that procedure for right-angled
Coxeter groups.  The weighted systems of Section~4 are handled by the
same principle, with the zero edge inner products replaced by the
prescribed Coxeter-angle equations.

Let \(\mathcal G\) be a finite graph on \(N\) vertices, and let
\(W_{\mathcal G}\) be its right-angled Coxeter group, as in
Subsection~\ref{subsection:group}.  For a target dimension \(d\), set
\[
 L_d=\operatorname{diag}(1,\ldots,1,-1).
\]
Thus \(L_5=L\) in the notation of Section~\ref{sec:preliminaries}.  In
this section, fix an ordering
\(V(\mathcal G)=\{v_1,\ldots,v_N\}\) and write
\(\nu_i=n_{v_i}\).  Thus the matrix \(\nu\) below has the normalized pole
vectors as columns; the implementation stores its transpose, with
\(\nu_i^{\mathsf T}\) as the \(i\)-th row.

\begin{definition}
A \emph{\(d\)-dimensional normal matrix} for \(\mathcal G\) is a matrix
\[
 \nu=(\nu_1\ \cdots\ \nu_N)
 \in\operatorname{Mat}_{(d+1)\times N}(\mathbb R)
\]
whose columns satisfy
\[
 \begin{cases}
  \nu_i^{\mathsf T}L_d\nu_j = 1,&i=j,\\
  \nu_i^{\mathsf T}L_d\nu_j = 0,&\{i,j\}\in E(\mathcal G),\\
  \nu_i^{\mathsf T}L_d\nu_j < -1,&i\ne j\text{ and }\{i,j\}\notin E(\mathcal G).
 \end{cases}
\]
\end{definition}

For the geometric applications below, we also require a compatible
orientation: there must be a point \(O_d\in\mathbf H^d\) such that
\[
 \langle O_d,\nu_i\rangle<0
 \qquad(1\leq i\leq N).
\]
For the flag-no-square nerves used above, a normal matrix with such an
orientation is precisely the algebraic input required by
Lemmas~\ref{lemma:intersection} and~\ref{lemma:connected}.  The geometric
argument of Subsection~\ref{subsec:discrete} then produces the discrete,
faithful, convex cocompact reflection representation.  Accordingly, the
task here is only to find and certify the normal matrix; the geometric
argument will not be repeated.

We call the resulting two-stage pipeline the \emph{RACG hyperbolic
realization-and-certification procedure} (RACG-HRCP).  A successful run
returns a rigorous certificate.  An unsuccessful run is merely
inconclusive and gives no obstruction to the existence of a
realization.  All applications in this paper have \(d=5\).

\subsection{Numerical search}\label{subsec:NS}

Since floating-point calculation does not guarantee exact satisfaction
of the equality constraints
\[
\nu_i^{\mathsf T} L_d \nu_i=1
\quad\text{and}\quad
\nu_i^{\mathsf T} L_d \nu_j=0,
\]
we first search for an approximate solution \(\hat \nu\), i.e.,
\[
    \begin{cases}
        {\hat \nu}_i^{\mathsf T} L_d {\hat \nu}_j \approx 1 & \text{if } i = j\\
        {\hat \nu}_i^{\mathsf T} L_d {\hat \nu}_j \approx 0 & \text{if } \{i,j\} \in E(\mathcal G) \\
        {\hat \nu}_i^{\mathsf T} L_d {\hat \nu}_j < -1 & \text{if } i \neq j \text{ and } \{i,j\} \notin E(\mathcal G)
    \end{cases}.
\]
This serves as the input to the certified verification step.

\subsubsection{Riemannian optimization}\label{subsubsec:RO}

Riemannian optimization is the problem of finding
\[
    x^* = \underset{x\in\mathcal M}{\operatorname{argmin}} \ f(x)
\]
where \((\mathcal M, g)\) is a Riemannian manifold and \(f:\mathcal M\to\mathbb R\) is a smooth function.
As in the Euclidean case, several optimization algorithms have
Riemannian analogues. For instance, Riemannian gradient descent
replaces the Euclidean update \(x \leftarrow x - \eta \nabla f(x)\) with
\[
    x \leftarrow \exp_x\bigl(-\eta\, \operatorname{grad} f(x)\bigr),
\]
where \(\operatorname{grad} f(x)\in T_x\mathcal M\) is the Riemannian
gradient of \(f\) at \(x\) and \(\exp_x: T_x\mathcal M \to \mathcal M\) is
the exponential map. In practice, \(\exp_x\) is often replaced by a
retraction \(R_x: T_x\mathcal M \to \mathcal M\), a computationally
cheaper first-order approximation to the exponential map satisfying
\(R_x(0)=x\) and \(\mathrm dR_x(0)=\operatorname{id}\). Stochastic
gradient descent and Adaptive Moment Estimation (Adam) admit similar
Riemannian analogues, and are implemented in \texttt{geoopt}
\cite{KKK:2020}, a PyTorch-based library for Riemannian
optimization. In addition to a collection of standard manifolds,
\texttt{geoopt} allows one to implement custom manifolds by
specifying the tangent space projection, the Riemannian metric, and a
retraction.

In our setting, we define the constraint manifold
\[
    \mathcal M = \left\{x=(x_0,\dots,x_d)\in\mathbb R^{d+1} \;\middle|\; x^{\mathsf T} L_d x = x_0^2+\cdots+x_{d-1}^2-x_d^2 = 1\right\}
\]
with the metric induced from \(\mathbb R^{d+1}\).
Each normal vector \(\nu_i\) is constrained to lie on \(\mathcal M\).
Hence the full optimization variable
\[
    \nu=(\nu_1,\dots,\nu_N)
\]
belongs to the product manifold \(\mathcal M^N\).
Since the diagonal conditions \(\nu_i^{\mathsf T} L_d \nu_i=1\) are enforced by the manifold constraint \(\nu_i\in\mathcal M\), it remains to impose the pairwise conditions
\[
    \begin{cases}
        \nu_i^{\mathsf T} L_d \nu_j=0 & \text{if } \{i,j\}\in E(\mathcal G),\\
        \nu_i^{\mathsf T} L_d \nu_j<-1 & \text{if } i \neq j \text{ and } \{i,j\}\notin E(\mathcal G).
    \end{cases}
\]
In the implementation, \(\mathcal M\) is represented as a custom manifold in \texttt{geoopt}, and the optimization is performed over \(N\) points of \(\mathcal M\).

To encourage these conditions, we minimize the objective function
\[
    f(\nu) = k_1 \cdot \frac{\sum_{\{i, j\} \in E(\mathcal G)} \sqrt{\bigl(\nu_i^{\mathsf T} L_d \nu_j\bigr)^2 + \epsilon}}{\max\{|E(\mathcal G)|, 1\}}  + k_2 \cdot \frac{\sum_{\{i, j\} \notin E(\mathcal G)} \operatorname{softplus}\bigl(k_3\bigl(\nu_i^{\mathsf T} L_d \nu_j + 1\bigr)\bigr)}{\max\{N(N-1)/2 - |E(\mathcal G)|, 1\}}
\]
where \(\operatorname{softplus}(x) = \log(1 + e^x)\) and \(k_1, k_2, k_3, \epsilon\) are positive real numbers.
The first term penalizes violations of the edge-orthogonality
conditions and attains its minimum when the corresponding inner
products vanish.  The positive constant \(\epsilon\) merely smooths the
absolute-value penalty and adds a constant offset at that minimum.
For a non-edge \(\{i, j\} \notin E(\mathcal G)\), the second term penalizes positive values of \(\nu_i^{\mathsf T} L_d \nu_j + 1\), and hence encourages the inequality \(\nu_i^{\mathsf T} L_d \nu_j < -1\).
Thus minimizing \(f\) drives the edge inner products toward zero and the
non-edge inner products below \(-1\).

Moreover, both the value and derivative of \(\operatorname{softplus}(x)\) approach zero as \(x \to -\infty\).
Thus, once \(\nu_i^{\mathsf T} L_d \nu_j\) is sufficiently smaller than \(-1\), the penalty exerts only a small additional force toward decreasing it further.

In this paper, we use \(k_1 = 70\), \(k_2 = 20\), \(k_3 = 100\), and \(\epsilon = 10^{-6}\), which were selected empirically from several parameter choices.

We minimize \(f\) using Riemannian Adam.
Since the optimization problem is nonconvex, the result is sensitive to the initial point and the algorithm may converge to a poor local minimum.
We therefore perform multiple runs from different initial points until we obtain a candidate whose constraint residuals are sufficiently small.

\subsubsection{ADMM}\label{subsubsec:ADMM}

The alternating direction method of multipliers (ADMM) is an optimization
method for problems with separable structure and coupling constraints.

A basic form of ADMM applies to convex problems of the form
\[
    \underset{x\in\mathbb R^n,\, z\in\mathbb R^m}{\operatorname{argmin}} \ f(x)+g(z)
    \qquad
    \text{subject to}
    \qquad
    Ax+Bz=c,
\]
where \(f:\mathbb R^n\to\mathbb R \cup \{+\infty\}\) and \(g:\mathbb R^m\to\mathbb R \cup \{+\infty\}\)
are closed, proper, convex, and \(A\in\mathbb R^{\tau\times n}\), \(B\in\mathbb
R^{\tau\times m}\), \(c\in\mathbb R^\tau\) are fixed.
The algorithm consists of the iterations
\begin{align*}
    x^{k+1} &:= \underset{x}{\operatorname{argmin}}\ \biggl(f(x) + (\rho/2)\|Ax+Bz^k-c+u^k\|_2^2\biggr) \\
    z^{k+1} &:= \underset{z}{\operatorname{argmin}}\ \biggl(g(z) + (\rho/2)\|Ax^{k+1}+Bz-c+u^k\|_2^2\biggr) \\
    u^{k+1} &:= u^k + Ax^{k+1}+Bz^{k+1}-c
\end{align*}
for some positive real number \(\rho\).

For nonconvex problems, ADMM need not converge; in this setting it
is best regarded as a heuristic local method rather than a solver with
guaranteed convergence; see \cite{Boyd:2011} for a detailed account.

In this stage we optimize the Gram matrix
\(Q=\nu^{\mathsf T}L_d\nu\) directly rather than the pole matrix
\(\nu\).  The constraints become
\[
    \begin{cases}
        Q_{ij} = 1 & \text{if } i = j,\\
        Q_{ij} = 0 & \text{if } \{i,j\} \in E(\mathcal G),\\
        Q_{ij} <-1 & \text{if } i \neq j \text{ and } \{i,j\} \notin E(\mathcal G),
    \end{cases}
\]
and
\begin{align*}
    &Q^{\mathsf T} = Q, \\
    &\operatorname{sig}(Q) = (s,1,N-s-1),
\end{align*}
for some \(s \leq d\) where the entries of \(\operatorname{sig}(Q)\) denote the numbers of positive, negative, and zero eigenvalues, respectively.

Define
\begin{align*}
    \mathcal C_{elt} &= \left\{Q \in \operatorname{Mat}_{N \times N}(\mathbb R)\ \middle| \ \begin{aligned}
        Q_{ij} &= 1 && \text{ if } i = j, \\
        Q_{ij} &= 0 && \text{ if } \{i,j\} \in E(\mathcal G), \\
        Q_{ij} &\leq -1 && \text{ if } i \neq j \text{ and } \{i,j\} \notin E(\mathcal G)
    \end{aligned} \right\}, \\
    \mathcal C_{sig} &= \left\{Q \in \operatorname{Mat}_{N \times N}(\mathbb R)\ \middle| \ \begin{aligned}
        &Q^{\mathsf T} = Q, \\
        &\exists\ s \in \{0, \dots, d\}, t \in \{0, 1\} \\
        &\operatorname{sig}(Q) = (s, t, N - s - t)
    \end{aligned}
    \right\},
\end{align*}
and 
\[
    f_{elt}(Q) := \begin{cases}
        0 & \text{if } Q \in \mathcal C_{elt} \\
        \infty & \text{if } Q \notin \mathcal C_{elt}
    \end{cases}, \qquad
    g_{sig}(Q) := \begin{cases}
        0 & \text{if } Q \in \mathcal C_{sig} \\
        \infty & \text{if } Q \notin \mathcal C_{sig}
    \end{cases}.
\]
The problem may therefore be written as
\[
    \underset{X,Z \in \operatorname{Mat}_{N \times N}(\mathbb{R})}{\operatorname{argmin}} \ f_{elt}(X)+g_{sig}(Z)
    \qquad
    \text{subject to}
    \qquad
    X-Z=0
\]
and the iteration step is
\begin{align*}
    X^{k+1} &:= \underset{X \in \operatorname{Mat}_{N \times N}(\mathbb{R})}{\operatorname{argmin}}\ \biggl(f_{elt}(X) + (\rho/2)\|X-Z^k+U^k\|_F^2\biggr) \\
    Z^{k+1} &:= \underset{Z \in \operatorname{Mat}_{N \times N}(\mathbb{R})}{\operatorname{argmin}}\ \biggl(g_{sig}(Z) + (\rho/2)\|X^{k+1}-Z+U^k\|_F^2\biggr) \\
    U^{k+1} &:= U^k + X^{k+1}-Z^{k+1}
\end{align*}
where \(\|Q\|_F = \sqrt{\operatorname{tr}(Q^{\mathsf T} Q)}\) denotes the Frobenius norm, i.e., the Euclidean norm on \(\operatorname{Mat}_{N \times N}(\mathbb R)\cong\mathbb R^{N^2}\).

\begin{fact}\label{indicator}
    Let \(\mathcal C \subseteq \mathbb R^n\) be a nonempty closed set and
    \[
        i_{\mathcal C} (x) = \begin{cases}
            0 & \text{if } x \in \mathcal C \\
            \infty & \text{if } x \notin \mathcal C
        \end{cases}
    \]
    be the indicator function.  Then
    \[
        \underset{x}{\operatorname{argmin}}\ \bigl(i_{\mathcal{C}}(x) + (\rho/2) h(x)\bigr) = \underset{x \in \mathcal C}{\operatorname{argmin}}\ h(x)
    \]
    for any \(h: \mathbb R^n \rightarrow \mathbb R\).
    Hence the minimizer is independent of the value of \(\rho>0\).
\end{fact}

The following lemma gives the explicit calculation.
\begin{lemma} \label{lem:admm-updates}
    Let \(X^k, Z^k, U^k\) be symmetric \(N \times N\) matrices.  Then
    \begin{enumerate}
        \item \(X^{k+1}\) is a symmetric matrix such that
        \[
            X^{k+1}_{ij} = \begin{cases}
                1 & \text{if } i = j \\
                0 & \text{if } \{i, j\} \in E(\mathcal G) \\
                \min \bigl\{(Z^k-U^k)_{ij}, -1\bigr\} & \text{if } i \neq j \text{ and } \{i, j\} \notin E(\mathcal G)
            \end{cases}.
        \]
        
        \item \(Z^{k+1}\) is a symmetric matrix such that
        \[
            Z^{k+1} = \min\{\lambda_1,0\}\, v_1 v_1^{\mathsf T} + \sum_{i=N-d+1}^{N} \max\{\lambda_i,0\}\, v_i v_i^{\mathsf T},
        \]
        where \(v_1,\dots,v_N\) are orthonormal eigenvectors of \(X^{k+1}+U^k\) corresponding to eigenvalues \(\lambda_1\le\cdots\le\lambda_N\).

        \item \(U^{k+1}\) is symmetric.
    \end{enumerate}
\end{lemma}

\begin{proof}
    Since \(\mathcal{C}_{elt}\) is convex, \(X^{k+1}\) is the projection of \(Z^k-U^k\) onto \(\mathcal C_{elt}\) as noted in \cite[\S 4.1]{Boyd:2011}, which directly proves (1).
    
    For (2), we follow the idea of \cite{EY:1936}.
    By Fact~\ref{indicator}, the minimizer \(Z\) lies in \(\mathcal C_{sig}\);
    in particular \(Z\) is symmetric, so we may write
    \(Z = PDP^{\mathsf T}\) for an orthogonal matrix \(P\) and a diagonal matrix
    \(D = \operatorname{diag}(\mu_1,\dots,\mu_N)\) whose entries are the
    eigenvalues of \(Z\), subject only to the signature condition
    \(D \in \mathcal C_{sig}\).
    We therefore minimize over the pair \((P, D)\) in two stages: we first fix
    \(D\) and determine the optimal \(P\), and then minimize over \(D\).

    Fix \(D\), so that \(Z\) is determined by \(P\) alone. Let
    \(Y = X^{k+1} + U^k\),
    \[
        l(P) = \|Y-PDP^{\mathsf T}\|_F^2 = \operatorname{tr}\bigl((Y-PDP^{\mathsf T})^2\bigr),
    \]
    and
    \[
        P_0 = \underset{P \in O(N)}{\operatorname{argmin}}\ l(P).
    \]
    Since \(l\) has minimum at \(P_0\),
    \[
        dl_{P_0}(S) = 0
    \]
    for all
    \[
        S \in T_{P_0}O(N) = \left\{S \in \operatorname{Mat}_{N \times N}(\mathbb{R}) \ \middle| \ S^{\mathsf T} P_0 + P_0^{\mathsf T} S = 0\right\}.
    \]
    Let \(S_0 = P_0^{\mathsf T}S\).  Then \(S_0\) is skew-symmetric.
    Since \(P_0^{\mathsf T}\) is invertible, \(S_0\) can be an arbitrary skew-symmetric matrix.
    Thus
    \begin{align*}
        dl_{P_0}(S) &= 2 \operatorname{tr}\bigl((Y-P_0 D P_0^{\mathsf T})(-SDP_0^{\mathsf T}-P_0DS^{\mathsf T})\bigr) \\
        &= 2 \operatorname{tr}\bigl((D - P_0^{\mathsf T} Y P_0)(P_0^{\mathsf T} SD + DS^{\mathsf T} P_0)\bigr) \\
        &= 2 \operatorname{tr}\bigl((D - P_0^{\mathsf T} Y P_0)(S_0 D - DS_0)\bigr) \\
        &= 2 \operatorname{tr}(DS_0D-D^2S_0-P_0^{\mathsf T} Y P_0 S_0 D+P_0^{\mathsf T} Y P_0DS_0) \\
        &= 2 \operatorname{tr}(P_0^{\mathsf T} Y P_0DS_0 - P_0^{\mathsf T} Y P_0 S_0 D) \\
        &= 2 \operatorname{tr}(P_0^{\mathsf T} Y P_0DS_0 - DP_0^{\mathsf T} Y P_0 S_0) \\
        &= 2 \operatorname{tr}([P_0^{\mathsf T} Y P_0, D]\;S_0) = 0.
    \end{align*}
    Hence \([P_0^{\mathsf T}YP_0,D]\) must be symmetric.
    However, since \(P_0^{\mathsf T} Y P_0\) and \(D\) are both symmetric, \([P_0^{\mathsf T} Y P_0, D]\) is skew-symmetric.
    Therefore, \([P_0^{\mathsf T} Y P_0, D]\) is zero and \(P_0^{\mathsf T} Y P_0\), \(D\) commute.
    Since both are symmetric, there is an orthogonal matrix \(R\) preserving
    each eigenspace of \(D\) such that \(R^{\mathsf T} (P_0^{\mathsf T} Y P_0) R\) is
    diagonal. Such an \(R\) satisfies \(RDR^{\mathsf T} = D\), so \(l(P_0R) = l(P_0)\)
    and \(P_0R\) is again a minimizer of \(l\). Replacing \(P_0\) by \(P_0R\),
    we may assume that \(P_0^{\mathsf T} Y P_0\) is diagonal.

    Furthermore, permuting the columns of \(P_0\) and the diagonal entries of
    \(D\) accordingly changes neither the value of \(l\) nor the membership
    \(D \in \mathcal C_{sig}\), since \(\mathcal C_{sig}\) constrains only the
    numbers of positive, negative, and zero eigenvalues. Hence we may further
    assume that the columns of
    \[
        P_0 = \begin{bmatrix}
            | & & | \\
            v_1 & \cdots & v_N\\
            | & & | \\
        \end{bmatrix}
    \]
    are orthonormal eigenvectors of \(Y\) with \(Y = P_0 \Lambda P_0^{\mathsf T}\),
    where \(\Lambda = \operatorname{diag}(\lambda_1,\dots,\lambda_N)\) and
    \(\lambda_1 \le \cdots \le \lambda_N\).
    
    Now, the Frobenius norm becomes
    \[
        \|Y-Z\|_F^2 = \|\Lambda-D\|_F^2 = \sum_{i = 1}^{N} (\lambda_i - \mu_i)^2,
    \]
    so the minimum occurs at
    \begin{align*}
        \mu_1 &= \min\{\lambda_1, 0\}, \\
        \mu_2 &= 0, \\
        & \ \; \vdots \\
        \mu_{N - d} &= 0, \\
        \mu_{N-d+1} &= \max\{\lambda_{N-d+1}, 0\}, \\
        & \ \; \vdots \\
        \mu_{N} &= \max\{\lambda_{N}, 0\}. \\
    \end{align*}
    
    Assertion~(3) follows immediately because
    \(U^k\), \(X^{k+1}\), and \(Z^{k+1}\) are symmetric.
\end{proof}

\begin{remark}
    The Gram matrix of a \(d\)-dimensional normal matrix satisfies the strict
    inequality \(Q_{ij} < -1\) for \(\{i,j\}\notin E(\mathcal G)\), and has
    exactly one negative eigenvalue. Neither of these conditions defines a
    closed set, so the corresponding Euclidean projections need not exist.
    We therefore relax both: \(\mathcal C_{elt}\) allows \(Q_{ij} \le -1\),
    and \(\mathcal C_{sig}\) allows \(t = 0\) as well as \(t = 1\). Both
    relaxations replace the original condition by its closure, so a matrix
    satisfying the original conditions still lies in \(\mathcal C_{elt} \cap
    \mathcal C_{sig}\).
\end{remark}

\subsubsection{Combining the two stages} \label{subsubsec:twostep}
Each of the two methods above has a complementary weakness, which
motivates using them together rather than separately.

In our computations, Riemannian optimization makes rapid initial
progress toward feasibility, but often stalls before reaching the
precision needed for certification.  ADMM can then improve the residuals
because its steps impose the constraints by projection rather than by a
gradient update.  Its weakness lies elsewhere: the signature projection
is nonconvex, so no general convergence theorem applies, and ADMM is
unreliable from an arbitrary starting point.

Because these two failure modes are complementary, we combine the
methods in sequence: Riemannian optimization is run first, to quickly
reach a neighborhood of feasibility, and its output is then refined
with ADMM to reach the precision required for certification.

Recall that \(\mathcal C_{elt}\) and \(\mathcal C_{sig}\) relax the
strict conditions \(Q_{ij}<-1\) and \(\operatorname{sig}(Q)=(s,1,N-s-1)\)
to their closures, so as to make the projections in
Lemma~\ref{lem:admm-updates} well-defined. Since Riemannian optimization
already drives the approximate solution close to satisfying the strict
conditions before ADMM is applied, we expect the non-edge entries to
remain strictly below \(-1\) and the smallest eigenvalue to remain
negative throughout the ADMM iterations, so that the strict conditions
continue to hold in practice despite this relaxation.

Concretely, writing \(\hat Q = \hat\nu^{\mathsf T} L_d \hat\nu\) for the Gram
matrix of the approximate solution obtained in \S\ref{subsubsec:RO}, we
initialize ADMM at
\[
    X^0 = Z^0 = \hat Q, \qquad U^0 = 0,
\]
which are symmetric. By Lemma~\ref{lem:admm-updates}, symmetry is
preserved under each update, so the hypothesis of the lemma holds at
every iteration.

We emphasize that no convergence guarantee is claimed for either stage
in isolation, or for their combination; this is immaterial to the
soundness of our overall method, since the correctness of a verified
output is guaranteed independently by the certification procedure of
Section~\ref{subsec:verification}, regardless of how \(\hat\nu\) was
obtained. Failure of this pipeline to converge (or convergence to an
infeasible point) simply results in an inconclusive certification,
prompting a restart from a new random seed.

\subsection{Rigorous verification} \label{subsec:verification}

We use the existence criterion of Chen and Womersley~\cite{CW:2006} for
systems of underdetermined equations.

Let \(c:\mathbb R^n\to\mathbb R^m\) be a continuously differentiable
function with \(m<n\), and suppose that \(\hat x\in\mathbb R^n\) is an
approximate solution of \(c(x)=0\) (i.e., \(c(\hat x)\approx 0\)) at
which the Jacobian \(c'(\hat x)\) has full row rank. Then there exists
an index set \(\mathcal B\subseteq\{1,\dots,n\}\) with \(|\mathcal
B|=m\) such that \(c'_{\mathcal B}(\hat x)\), the matrix obtained by
selecting the columns indexed by \(\mathcal B\) from \(c'(\hat x)\), is
invertible. Let \(\mathcal N=\mathcal B^c\). For two nonnegative
numbers \(r_1\) and \(r_2\), define the convex set
\[
    X = \{x \in \mathbb R^n \mid \|x_{\mathcal B} - \hat x_{\mathcal B}\| \leq r_1, \|x_{\mathcal N} - \hat x_{\mathcal N}\| \leq r_2\}
\]
where \(x_I\) denotes the vector obtained by selecting the entries of \(x\) indexed by \(I\subseteq\{1,\dots,n\}\) and \(\|\cdot\|\) is a fixed vector space norm.

\begin{thm}[Chen--Womersley \cite{CW:2006}]\label{thm:CW}
    Suppose that there exists \(K > 0\) satisfying
    \[
        \|c'_{\mathcal B}(x) - c'_{\mathcal B}(\hat x)\| \leq K \|x - \hat x\|
    \]
    for all \(x \in X\).
    If
    \[
        \|c'_{\mathcal B}(\hat x)^{-1} c(\hat x)\| + \|c'_{\mathcal B}(\hat x)^{-1}\|\biggl(\frac{1}{2} K (r_1 + r_2) r_1 + \max_{x \in X} \|c'_{\mathcal N}(x)\| r_2\biggr) \leq r_1,
    \]
    there exists a solution of \(c(x) = 0\) in \(X\).
\end{thm}

\begin{remark}
    This also holds for \(m=n\), with \(\mathcal N=\emptyset\), in which case
    the theorem coincides with the classical Newton--Kantorovich theorem
    \cite{OR:1970} (cf. \cite[Remark~2.1]{CW:2006}).
\end{remark}

Let \(\hat \nu \in \operatorname{Mat}_{(d+1) \times N}(\mathbb R)\) be an approximate \(d\)-dimensional normal matrix, i.e.,
\[
    \begin{cases}
        \hat \nu_i^{\mathsf T} L_d \hat \nu_j \approx 1 & \text{if } i = j\\
        \hat \nu_i^{\mathsf T} L_d \hat \nu_j \approx 0 & \text{if } \{i,j\} \in E(\mathcal G) \\
        \hat \nu_i^{\mathsf T} L_d \hat \nu_j < -1 & \text{if } \{i,j\} \notin E(\mathcal G)
    \end{cases}.        
\]
Define \(F: \operatorname{Mat}_{(d+1)\times N}(\mathbb R) \rightarrow \mathbb R^{N + |E(\mathcal G)|}\) as
\[
    F(v) = \bigl(v_1^{\mathsf T} L_d v_1 - 1, \dots, v_N^{\mathsf T} L_d v_N - 1,\ v_{i_1}^{\mathsf T} L_d v_{j_1}, \dots, v_{i_{|E(\mathcal G)|}}^{\mathsf T} L_d v_{j_{|E(\mathcal G)|}}\bigr)
\]
where \(\{i_l, j_l\} \in E(\mathcal G)\), \(i_l < j_l\).
Then \(F\) is continuously differentiable, and a matrix
\(\nu\in\operatorname{Mat}_{(d+1)\times N}(\mathbb R)\) is an exact solution
of \(F(\nu)=0\) if and only if
\[
\begin{cases}
    \nu_i^{\mathsf T} L_d \nu_j = 1 & \text{if } i = j\\
    \nu_i^{\mathsf T} L_d \nu_j = 0 & \text{if } \{i,j\} \in E(\mathcal G)
\end{cases}.
\]
We want to check the existence of a solution \(\nu\) of \(F\) near an approximate solution \(\hat \nu\) of \(F\).

From now on, we identify \(\operatorname{Mat}_{(d+1)\times N}(\mathbb R)\) with \(\mathbb R^{N(d+1)}\) as \(v = (v_1^{\mathsf T}, \dots, v_N^{\mathsf T})\).

We first record a global Lipschitz bound for the Jacobian.  The simple
constant below is sufficient because \(F\) is quadratic.

\begin{lemma} \label{lipschitzbound}
    For all \(u, v \in \mathbb R^{N(d+1)}\),
    \[
        \|F'(u) - F'(v)\|_\infty \leq 2(d+1) \|u - v\|_\infty
    \]
    where \(\| \cdot \|_\infty\) denotes the \(\infty\)-norm and its induced matrix norm.
\end{lemma}
\begin{proof}
    Recall that for a matrix \(A\), 
    \[
    \|A\|_{\infty} = \max_{i} \sum_{j} |A_{ij}|.
    \]
    Let
    \[
        f_i(v) = v_i^{\mathsf T} L_d v_i, \quad g_l(v) = v_{i_l}^{\mathsf T} L_d v_{j_l}
    \]
    for \(1 \leq i \leq N\), \(1 \leq l \leq |E(\mathcal G)|\).
    Thus
    \begin{align*}
        f_i'(v) &= (0, \dots, 0, 2v_i^{\mathsf T} L_d, 0, \dots, 0), \\
        g_l'(v) &= (0, \dots, 0, v_{j_l}^{\mathsf T} L_d, 0, \dots, 0, v_{i_l}^{\mathsf T} L_d, 0, \dots, 0).
    \end{align*}
    Write \(u_i=(u_{i,1},\dots,u_{i,d+1})^{\mathsf T}\) and
    \(v_i=(v_{i,1},\dots,v_{i,d+1})^{\mathsf T}\).  Then
    \begin{align*}
        &\|F'(u) - F'(v)\|_\infty \\
        =& \max_{\substack{1 \leq i \leq N \\ 1 \leq l \leq |E(\mathcal G)|}} \Bigg\{\sum_{q = 1}^{d + 1} 2 |u_{i, q} - v_{i, q}|,\ \sum_{q = 1}^{d + 1}\bigl(|u_{j_l, q} - v_{j_l, q}| + |u_{i_l, q} - v_{i_l, q}|\bigr) \Bigg\} \\
        \leq& \max\bigl\{2(d+1)\|u-v\|_\infty,\ 2(d+1)\|u-v\|_\infty\bigr\} \\
        =& 2(d+1)\|u - v\|_{\infty}.
    \end{align*}
\end{proof}

We now specialize Theorem~\ref{thm:CW} to the constraint map \(F\).

\begin{thm} \label{thm:criterion}
    Suppose \(N + |E(\mathcal G)| \leq N (d + 1)\) and \(F'(\hat \nu)\) has full row rank.
    Pick \(\mathcal{B} \subseteq \{1, \dots, N(d+1)\}\) with \(|\mathcal B| = N + |E(\mathcal G)|\) such that \(F'_\mathcal{B}(\hat \nu)\) is invertible, and let \(\mathcal{N} = \mathcal{B}^c\).
    For
    \[
        \eta = \|F'_\mathcal{B}(\hat \nu)^{-1} F(\hat \nu)\|_\infty, \quad \beta = \|F'_\mathcal{B}(\hat \nu)^{-1}\|_\infty,
    \]
    if
    \[
        D = 1 - 4 \eta \beta (d + 1) \geq 0,
    \]
    then there is an exact solution \(\nu\) of \(F\) in
    \[
        X =
        \bigl\{
        v \in \mathbb{R}^{N(d+1)}
        \mid
        \|v_{\mathcal B}-\hat \nu_{\mathcal B}\|_\infty \le r_1,\;
        v_{\mathcal N}=\hat \nu_{\mathcal N}
        \bigr\}
    \]
    where
    \[
        r_1 = \frac{1 - \sqrt{D}}{2 \beta (d + 1)}.
    \]
\end{thm}

\begin{proof}
    For simplicity, we fix \(r_2 = 0\).
    Then the inequality in Theorem~\ref{thm:CW} becomes
    \[
        \frac{1}{2} \beta K r_1^2 - r_1 + \eta \leq 0.
    \]
    If there is a nonnegative \(r_1\) satisfying the inequality, then \(F\) admits an exact solution in
    \[
        X =
        \bigl\{
        v \in \mathbb{R}^{N(d+1)}
        \mid
        \|v_{\mathcal B}-\hat \nu_{\mathcal B}\|_\infty \le r_1,\;
        v_{\mathcal N}=\hat \nu_{\mathcal N}
        \bigr\}
    \]
    by Theorem~\ref{thm:CW}.
    This quadratic (in \(r_1\), with positive leading coefficient
    \(\beta K/2\)) has a real root if and only if its discriminant
    \[
        D = 1 - 2 \beta K \eta = 1 - 4 \eta \beta (d+1)
    \]
    is nonnegative, in which case the set of \(r_1\) satisfying the
    inequality is the interval between the two roots. Since the product
    of the roots is \(\eta/(\beta K/2) \ge 0\) and their sum is
    \(1/(\beta K/2) > 0\), both roots are nonnegative whenever they are
    real; hence \(D\ge0\) is exactly the condition for a nonnegative
    \(r_1\) to exist, and the tightest such bound is the smaller root
    \[
        r_1 = \frac{1 - \sqrt{D}}{2 \beta (d + 1)}.
    \]
\end{proof}

It remains to check the inequality
\[
    \nu_i^{\mathsf T} L_d \nu_j < -1 \quad\text{if } i \neq j \text{ and } \{i,j\} \notin E(\mathcal G).
\]
We verify this using ball arithmetic, implemented via \texttt{Arb}. Since
rigor is essential in this section, every computation except the
selection of \(\mathcal B\) is carried out in ball arithmetic, and we
report Verified only when the inequality holds for the entire ball,
i.e., for every point in its enclosure. (The selection of \(\mathcal B\)
itself need not be certified, since its validity is confirmed rigorously
by the successful ball-arithmetic inversion of \(F'_{\mathcal B}(\hat\nu)\)
in Theorem~\ref{thm:criterion}.)

By the definition of \(X\) in Theorem~\ref{thm:criterion}, the exact
solution \(\nu\) satisfies \(\nu_{\mathcal N} = \hat\nu_{\mathcal N}\) and
\(\|\nu_{\mathcal B}-\hat\nu_{\mathcal B}\|_\infty \le r_1\). We therefore
construct the ball matrix whose entries agree with \(\hat\nu\) at the
midpoint, with radius \(r_1\) at indices in \(\mathcal B\) and radius
\(0\) at indices in \(\mathcal N\).

If the ball evaluation of \(\nu_i^{\mathsf T} L_d \nu_j\) is contained in
\((-\infty,-1)\) for every \(\{i,j\}\notin E(\mathcal G)\), then the
exact solution \(\nu\) guaranteed by Theorem~\ref{thm:criterion} also
satisfies this inequality. Together with Theorem~\ref{thm:criterion},
this proves the existence of the required normal matrix.  To obtain the
oriented realization, we also evaluate
\(\langle O_d,\nu_i\rangle\) on the same ball matrix and require its upper
endpoint to be negative for every \(i\).  This final test is linear and is
independent of the existence criterion.

\subsection{Automated implementation}\label{subsec:implementation}
To automate the process, we do the following.

\begin{algorithm}[H]
\caption{Automated RACG-HRCP}\label{alg:racg-hrcp}
\KwIn{\(\texttt{num\_seeds}\); \(\texttt{prelude}\);
\(\texttt{max\_chunks}\); \(\texttt{max\_attempts}\), together with the
search tolerances and block sizes}
\KwOut{Verified or Inconclusive}
\(\texttt{seeds} \gets \texttt{num\_seeds}\) random seeds\;
\For{\(\texttt{seed} \in \texttt{seeds}\)}{
    \(\hat \nu \gets\) random initial point in \(\mathcal M^N\) from \(\texttt{seed}\)\;
    apply \(\texttt{prelude}\) blocks of ADMM to \(\hat\nu\)\;
    \While{\(\hat\nu\) has not met the feasibility thresholds}{
        \If{the RO budget is exhausted or the search has stalled}{
            skip to the next seed\;
        }
        \(\hat \nu \gets \texttt{RO}(\hat \nu)\)\;
    }
    polish by further RO and retain the admissible iterate with
    the largest non-edge margin\;
    \For{\(\texttt{attempt} \gets 1\) \KwTo
    \(\texttt{max\_attempts}\)}{
        \If{\(\texttt{Verify}(\hat \nu)=\texttt{Verified}\)}{
            \Return Verified\;
        }
        \If{\(\texttt{attempt}<\texttt{max\_attempts}\)}{
            refine \(\hat\nu\) by the prescribed number of ADMM blocks\;
        }
    }
}
\Return Inconclusive\;
\end{algorithm}

An optional ADMM prelude substantially improved the empirical success
rate in the computations reported here.  Because distinct random seeds
can occasionally be carried to the same iterate during this prelude, it
may also reduce the diversity of the subsequent search.  The prelude
length and the other search parameters are therefore recorded in each
certificate; they affect reproducibility and performance, but not the
validity of a successful certificate.

Our implementation \cite{RACGRealizer} provides three entry points.
Running \path{app_auto.py} executes
Algorithm~\ref{alg:racg-hrcp} automatically; \path{run.py} provides the
same pipeline from the command line, while \path{app_manual.py} allows
the user to perform the optimization and refinement stages
interactively.

\subsection{Scope of the common certification}
\label{subsec:scope-certification}

Although the finite realizations above arise from different
combinatorial ansatzes, their equality and strict-inequality conditions
fit a common interval-certification framework.  It is important to
distinguish this final certification from the method used to discover an
approximate solution: the latter may vary, whereas the validity of the
output is decided independently.

\begin{prop}\label{prop:common-certification}
Every finite numerical realization used in
Propositions~\ref{prop:certified-453},
\ref{prop:nonrightcertificate}, and~\ref{prop:asymmetric-certificate},
as well as the
finite candidates recorded in
Appendix~\ref{appendix:numerical-motivation}, admits a rigorous
verification within the certification framework of this section.
\end{prop}

\begin{proof}
For Proposition~\ref{prop:asymmetric-certificate}, this is precisely the
underdetermined inclusion test of Theorem~\ref{thm:criterion}, followed
by interval evaluation of every non-edge.  The complete ancillary data are
available in \cite{HCPd,RACGRealizer}, and the approximate Lorentz-unit
\(p=5\) pole matrix is displayed in
Appendix~\ref{appendix:asymmetric-data}.  In the
symmetric and weighted
constructions, gauge fixing reduces the relevant edge equations to
square systems.  The square case is the \(m=n\) specialization described
after Theorem~\ref{thm:CW}; equivalently, one may use the interval
Newton--Krawczyk form implemented in the ancillary programs
\path{verify_g453_arb.py}, \path{verify_nonright_arb.py}, and
\path{verify_sections32_33_arb.py}~\cite{HCPd}.  In every case the
existence of an exact zero and all strict non-edge inequalities are
verified with directed ball arithmetic.

The uniform families of Proposition~\ref{prop:uniform-smallp} are
different: their validity is proved analytically for all \(m\geq5\), so
no finite computation is used in their proof.  Individual members may
of course be checked by the same framework.
\end{proof}

%\appendix

%\section{Experimental motivation for the uniform construction}
\appendix

\section{Computational origin of the uniform ansatz}
\label{appendix:numerical-motivation}

We explain briefly how finite computations within the cyclic ansatz of
Subsection~\ref{subsec:construction} led to the normalized variables used
in the uniform construction.  The finite solutions considered here lie
in a parameter family different from that of
Proposition~\ref{prop:uniform-smallp}; in particular, their values of
\(a\) and \(b\) are not specializations of the parameters used there.
They are included only to describe the equation-reduction strategy and
the computational origin of the uniform ansatz.

After imposing the four edge equations
\[
\begin{aligned}
 \langle n(A_1),n(A_2)\rangle
 &=\langle n(A_1),n(B_1)\rangle
  =\langle n(A_1),n(B_2)\rangle\\
 &=\langle n(B_1),n(B_2)\rangle=0,
\end{aligned}
\]
the coordinates \(x_a,x_b,y_b,z_b\) are determined by \(a\) and \(b\).
Thus one may take
\[
 a,\qquad b,\qquad c,\qquad x_c
\]
as the remaining free parameters.  For fixed values of these
parameters, the three equations
\[
\begin{aligned}
 \langle n(B_1),n(C_1)\rangle
 &=\langle n(B_1),n(C_2)\rangle\\
 &=\langle n(C_1),n(C_2)\rangle=0
\end{aligned}
\]
determine \(y_c,z_c,w_c\), whenever the corresponding real solution
exists.  It then remains to verify spacelikeness of the pole vectors and
strict hyperparallelism for every non-edge pair.

For \(p=2\), the finite search was carried out using
\[
 a=-\sin\frac{\pi}{2m},
 \qquad
 b=\sin\frac{\pi}{2m}.
\]
With this choice, the loci on which selected non-edge pairs become
asymptotic bound a feasible region in the \((c,x_c)\)-plane.  Interior
points of this region give the numerical candidates.  Some
representative centers are listed in
Table~\ref{tab:finite-candidates}; the final row records one example for
\(p=3\).

\begin{table}[htbp]
\centering
\small
\setlength{\tabcolsep}{4pt}
\begin{tabular}{cccccc}
\toprule
\(p\) & \(m\) & \(a\) & \(b\) & \(c\) & \(x_c\)\\
\midrule
\(2\) & \(5\)
& \(-\sin(\pi/10)\) & \(\sin(\pi/10)\)
& \(0.799017656466\) & \(0.122273943505\)\\
\(2\) & \(9\)
& \(-\sin(\pi/18)\) & \(\sin(\pi/18)\)
& \(0.605581569857\) & \(0.605581569857\)\\
\(2\) & \(14\)
& \(-\sin(\pi/28)\) & \(\sin(\pi/28)\)
& \(0.433670513698\) & \(0.819472588890\)\\
\(3\) & \(6\)
& \(-0.4\) & \(0.05\)
& \(0.642452725867\) & \(0.348936391419\)\\
\bottomrule
\end{tabular}
\caption{Representative centers arising from the finite computations.}
\label{tab:finite-candidates}
\end{table}
\tablegap
The decimal entries above are numerical
centers rather than exact solutions.  The ancillary program
\path{verify_sections32_33_arb.py}, archived in the supplementary
repository~\cite{HCPd}, certifies all candidates with
\[
 (p,m)=(2,m),\qquad 5\leq m\leq14,
\]
together with the candidate for \((p,m)=(3,6)\).  For each candidate, a
Krawczyk inclusion~\cite{Krawczyk:1969} proves the existence and local
uniqueness of an exact solution near the stored center.  Arb ball
arithmetic~\cite{Johansson:2017}, accessed through
\texttt{python-flint}~\cite{pythonflint}, then verifies spacelikeness,
all strict non-edge inequalities, and non-positivity of the
off-diagonal entries of the normalized Gram matrix.  It also certifies
the Gram inertia
\[
 (5,1,(2p+1)m-5).
\]
These finite certificates are not used in the proof of
Proposition~\ref{prop:uniform-smallp}.

The conceptual contribution of the finite calculations is the repeated
appearance of the normalized radial quantities
\[
 \frac{x_b^2+y_b^2}{1-b^2},
 \qquad
 \frac{x_c^2+y_c^2}{1-c^2}.
\]
Both the edge equations and the non-edge inequalities depend naturally
on these combinations rather than on the individual coordinates.  This
suggests prescribing the normalized radii directly when constructing a
family uniform in \(m\).

In the notation of the uniform construction, put
\[
 t=\frac{2\pi}{m},
 \qquad
 U=1-\rho_pt^2,
 \qquad
 V=1-\sigma_pt^2.
\]
The coordinates defined in
\eqref{eq:b-value}--\eqref{eq:yc-value} then satisfy
\[
 \frac{x_b^2+y_b^2}{1-b^2}=U,
 \qquad
 \frac{x_c^2+y_c^2}{1-c^2}=V.
\]
Thus the quantities \(U\) and \(V\) used in the uniform-separation
argument arise naturally from the finite equation reduction.  The
constants \(\rho_p\) and \(\sigma_p\) are chosen to give a uniform
positive margin in every non-edge inequality; they are not obtained by
fitting the finite data in Table~\ref{tab:finite-candidates}.

Recall that
\[
 1-b^2=R_p(m)^{-1}.
\]
Writing
\[
 (1-b^2)^{-1/2}=\cosh r,
 \qquad
 \Lambda_p(m)=\cosh s,
\]
the relation imposed by the \(B_1C_1\) and \(B_1C_2\) edge equations
becomes
\[
 (1-c^2)^{-1/2}=\cosh(r+s).
\]
Consequently, the hyperbolic addition formula gives
\[
 \mathcal M_p(m)
 =\sqrt{R_p(m)}\,\Lambda_p(m)
  +\sqrt{R_p(m)-1}\sqrt{\Lambda_p(m)^2-1}
 =\cosh(r+s),
\]
which is precisely the definition in \eqref{eq:calM}.  The finite and
uniform constructions therefore share the same equation-reduction
mechanism, although the closed-form uniform family is a distinct
solution within the cyclic ansatz.

\clearpage
\section{A complete asymmetric pole configuration}
\label{appendix:asymmetric-data}

For \(p=5\), the following table gives the complete approximate
Lorentz-unit pole matrix
\[
 V^{(0)}=(n_v^{(0)})_{v\in V(\mathcal G_5)}
\]
in the notation of Subsection~\ref{subsection:asysub}.  Each stored row
has Lorentz norm \(1\), up to floating-point roundoff.  The
ball-arithmetic certificate proves the existence of an
exact normal matrix in a neighborhood of \(V^{(0)}\).  The entries below
are rounded; the full-precision matrices for \(2\leq p\leq8\), together
with their certificates, are contained in the ancillary material
\cite{HCPd,RACGRealizer}.

\begingroup
\scriptsize
\setlength{\tabcolsep}{1.5pt}
\renewcommand{\arraystretch}{0.96}
\begin{longtable}{@{}lrrrrrr@{}}
\toprule
\(v\) & \(x_1\) & \(x_2\) & \(x_3\) & \(x_4\) & \(x_5\) & \(x_6\)\\
\midrule
\endfirsthead
\toprule
\(v\) & \(x_1\) & \(x_2\) & \(x_3\) & \(x_4\) & \(x_5\) & \(x_6\)\\
\midrule
\endhead
\midrule
\multicolumn{7}{r}{\emph{continued on the next page}}\\
\endfoot
\bottomrule
\endlastfoot
\(A_1\) & \(-0.8412743229\) & \(-0.9185812809\) & \(0.1215020206\) & \(-0.1325873731\) & \(-0.0468118659\) & \(0.7655504942\)\\
\(A_2\) & \(-0.3350660655\) & \(1.1795099357\) & \(-0.4138040208\) & \(-0.2053772843\) & \(-0.1917189701\) & \(0.8681490176\)\\
\(A_3\) & \(1.2618293329\) & \(-0.2752489854\) & \(0.0659044087\) & \(-0.0232377524\) & \(0.1082192127\) & \(0.8273874858\)\\
\(C_1\) & \(-0.0127144284\) & \(-0.0101870366\) & \(-0.9290145886\) & \(-1.5189084262\) & \(1.4972297798\) & \(2.1005031205\)\\
\(C_2\) & \(-0.2003208916\) & \(0.0208173399\) & \(-0.9817971564\) & \(0.8107452419\) & \(1.3813542909\) & \(1.6031016817\)\\
\(C_3\) & \(-0.0079569583\) & \(-0.1507297561\) & \(-0.6707311712\) & \(1.5110329086\) & \(-0.8279547338\) & \(1.5624956220\)\\
\(C_4\) & \(0.3265250655\) & \(-0.2258662292\) & \(-0.6863844835\) & \(-0.5217739594\) & \(-1.7326722025\) & \(1.7038658566\)\\
\(C_5\) & \(0.1818071825\) & \(0.3596708420\) & \(1.0479250947\) & \(-1.6438224536\) & \(-0.7279911652\) & \(1.8688732878\)\\
\(C_6\) & \(-0.0554673267\) & \(0.2745165678\) & \(1.2333630662\) & \(0.4589247211\) & \(-1.4249795946\) & \(1.6854670474\)\\
\(C_7\) & \(-0.1695192333\) & \(0.3570829860\) & \(0.9246770075\) & \(1.3498828187\) & \(0.6821591311\) & \(1.5161785191\)\\
\(C_8\) & \(-0.0463413338\) & \(-0.2345743613\) & \(0.8023703295\) & \(-0.4516848647\) & \(2.1513308904\) & \(2.1291347098\)\\
\(e_{A_1A_2}\) & \(-1.5380130351\) & \(0.3285685508\) & \(-0.3222242225\) & \(-0.3523765412\) & \(-0.1804934002\) & \(1.3168207671\)\\
\(e_{A_1A_3}\) & \(0.5332638824\) & \(-1.5898142094\) & \(0.2453605481\) & \(-0.1031906699\) & \(0.0719857291\) & \(1.3740129718\)\\
\(e_{A_1C_1}\) & \(-0.8534130988\) & \(-0.8722495646\) & \(-0.5908997672\) & \(-1.7233238653\) & \(0.8685586050\) & \(2.1360091175\)\\
\(e_{A_1C_2}\) & \(-0.9252866072\) & \(-1.0508318732\) & \(-1.2918349023\) & \(0.3274586313\) & \(1.0372578548\) & \(1.9525299892\)\\
\(e_{A_1C_3}\) & \(-0.7263217011\) & \(-1.4611934597\) & \(-1.0252281812\) & \(1.2817735874\) & \(-0.7305152110\) & \(2.2114064227\)\\
\(e_{A_1C_4}\) & \(-0.6060521724\) & \(-1.3025553949\) & \(-1.1054621858\) & \(-0.8146949480\) & \(-1.5100528600\) & \(2.2869158129\)\\
\(e_{A_1C_5}\) & \(-0.6313125379\) & \(-0.8044339567\) & \(0.8959008841\) & \(-1.8423036672\) & \(-0.6507050693\) & \(2.1600481001\)\\
\(e_{A_1C_6}\) & \(-0.7880200630\) & \(-1.0495995062\) & \(1.6343655939\) & \(0.5762326126\) & \(-1.3733495539\) & \(2.3689488510\)\\
\(e_{A_1C_7}\) & \(-1.3885611239\) & \(-0.4730438080\) & \(0.8178643351\) & \(1.3010136980\) & \(0.5812136079\) & \(1.9624526517\)\\
\(e_{A_1C_8}\) & \(-0.8609928978\) & \(-1.1877989461\) & \(1.3350364671\) & \(-0.9602209950\) & \(1.7667104805\) & \(2.6415502562\)\\
\(e_{A_2A_3}\) & \(1.1543279436\) & \(1.1257717097\) & \(-0.3767688272\) & \(-0.2460503825\) & \(-0.1166577238\) & \(1.3475679975\)\\
\(e_{A_2C_1}\) & \(-0.1761436852\) & \(0.9906304512\) & \(-1.2166286637\) & \(-1.8582750630\) & \(0.5947915099\) & \(2.3020694482\)\\
\(e_{A_2C_2}\) & \(-0.6897235879\) & \(1.2802386521\) & \(-1.8787032768\) & \(0.8031917577\) & \(0.9713256299\) & \(2.4965668658\)\\
\(e_{A_2C_3}\) & \(-0.0211383562\) & \(1.3503011384\) & \(-1.0914385608\) & \(1.4331340264\) & \(-0.9695427435\) & \(2.2380537070\)\\
\(e_{A_2C_4}\) & \(-0.0807326488\) & \(0.7438377993\) & \(-1.8441698554\) & \(-0.8173123965\) & \(-1.5514146791\) & \(2.4567584228\)\\
\(e_{A_2C_5}\) & \(-0.5655957990\) & \(1.4229924536\) & \(0.8674198674\) & \(-1.5698632205\) & \(-0.1394451434\) & \(2.1403595110\)\\
\(e_{A_2C_6}\) & \(-0.1299908159\) & \(1.7274657316\) & \(0.9433606665\) & \(0.1574784067\) & \(-1.4272029448\) & \(2.2254600664\)\\
\(e_{A_2C_7}\) & \(-0.5050137958\) & \(1.5541291654\) & \(0.4034178455\) & \(0.8613229204\) & \(0.7081631786\) & \(1.7539882028\)\\
\(e_{A_3C_1}\) & \(1.3772103283\) & \(-0.2239744123\) & \(-0.6524251213\) & \(-1.4462737403\) & \(1.4370855145\) & \(2.3514790819\)\\
\(e_{A_3C_2}\) & \(1.2408355256\) & \(-0.3613859310\) & \(-1.0538778454\) & \(0.6994997824\) & \(1.4633760585\) & \(2.1004048515\)\\
\(e_{A_3C_3}\) & \(1.3304194383\) & \(-0.5086201760\) & \(-0.9796278829\) & \(1.1117755706\) & \(-0.8339294172\) & \(1.9798647293\)\\
\(e_{A_3C_4}\) & \(1.6030984975\) & \(-0.6386282645\) & \(-1.0987274798\) & \(-1.0975556169\) & \(-1.2459875447\) & \(2.4376394773\)\\
\(e_{A_3C_5}\) & \(1.3933292304\) & \(-0.3116570492\) & \(0.8596875966\) & \(-1.8662561309\) & \(-0.2889894804\) & \(2.3117063158\)\\
\(e_{A_3C_6}\) & \(1.1896388910\) & \(-0.0661131577\) & \(1.3455522901\) & \(0.0440826448\) & \(-1.0197352612\) & \(1.8088465630\)\\
\(e_{A_3C_7}\) & \(1.2896890517\) & \(0.3556395162\) & \(1.1836450684\) & \(1.1655603204\) & \(0.6610361265\) & \(1.9965702053\)\\
\(e_{A_3C_8}\) & \(1.2980350577\) & \(-0.3929344447\) & \(1.2373875069\) & \(-0.5180475809\) & \(1.8555334455\) & \(2.4661301658\)\\
\(e_{C_1C_2}\) & \(-0.1362128994\) & \(0.5543342947\) & \(-1.5741524998\) & \(-0.5086881494\) & \(2.2244117239\) & \(2.6477476669\)\\
\(e_{C_1C_8}\) & \(-0.4524177615\) & \(0.4241863083\) & \(0.2069486258\) & \(-1.6694778349\) & \(2.4228320520\) & \(2.8433633222\)\\
\(e_{C_2C_3}\) & \(-0.1351687241\) & \(-0.1511818839\) & \(-1.4895662151\) & \(2.0053241138\) & \(0.4689808798\) & \(2.3454641167\)\\
\(e_{C_2C_7}\) & \(0.1881574267\) & \(-0.1350565558\) & \(0.0246954698\) & \(2.1615470351\) & \(2.0252340073\) & \(2.7978762885\)\\
\(e_{C_2C_8}\) & \(-0.6525857048\) & \(-0.0778244068\) & \(-0.0924610244\) & \(0.4061599391\) & \(2.9633956550\) & \(2.8960582670\)\\
\(e_{C_3C_4}\) & \(-0.2072703306\) & \(-0.3492960198\) & \(-1.2257482084\) & \(0.8507754957\) & \(-2.3487035092\) & \(2.6282417104\)\\
\(e_{C_3C_6}\) & \(0.1967260672\) & \(0.0785483172\) & \(0.5221856515\) & \(1.8289224306\) & \(-2.0583403316\) & \(2.6266463439\)\\
\(e_{C_3C_7}\) & \(0.0294364728\) & \(-0.0221018271\) & \(0.3539818740\) & \(2.7343131342\) & \(-0.1505544883\) & \(2.5740616026\)\\
\(e_{C_4C_5}\) & \(0.3144691459\) & \(0.3769863644\) & \(0.0631005886\) & \(-1.7843838589\) & \(-1.8756526402\) & \(2.4386655838\)\\
\(e_{C_4C_6}\) & \(-0.0903188973\) & \(-0.0390230165\) & \(0.5159040066\) & \(-0.1587114844\) & \(-2.9295121529\) & \(2.8076801870\)\\
\(e_{C_5C_6}\) & \(-0.4681061699\) & \(0.4315403123\) & \(2.0609347497\) & \(-1.0049376961\) & \(-1.6601182690\) & \(2.7237281271\)\\
\(e_{C_6C_7}\) & \(-0.2490180799\) & \(0.8090529828\) & \(1.8175493763\) & \(1.6611625696\) & \(-0.6820714318\) & \(2.4989487768\)\\
\(e_{C_7C_8}\) & \(-0.2953064107\) & \(0.0117692547\) & \(1.7289323108\) & \(0.6903958542\) & \(2.1314203788\) & \(2.6638600927\)\\
\(f_{A_1;A_2C_1}\) & \(-1.7801693046\) & \(-0.5969787055\) & \(-1.5302853737\) & \(-1.4988443947\) & \(0.2173342010\) & \(2.6759910925\)\\
\(f_{A_1;A_2C_2}\) & \(-1.9778086251\) & \(-0.4182035027\) & \(-1.3669382029\) & \(0.6249415532\) & \(0.3013504250\) & \(2.3316314463\)\\
\(f_{A_1;A_2C_4}\) & \(-1.9107568297\) & \(-0.5603496713\) & \(-1.1520628218\) & \(-0.8783064620\) & \(-1.7459396021\) & \(2.8481501887\)\\
\(f_{A_1;A_2C_5}\) & \(-1.9092217882\) & \(-0.0999780268\) & \(0.5648392788\) & \(-1.8791882844\) & \(-0.8366283615\) & \(2.6842992524\)\\
\(f_{A_1;A_2C_7}\) & \(-2.3480367026\) & \(0.1802104038\) & \(0.8514520553\) & \(0.3840621757\) & \(0.5744669838\) & \(2.3975485020\)\\
\(f_{A_1;A_3C_2}\) & \(0.0048009440\) & \(-1.9265064183\) & \(-0.9175746326\) & \(-0.1793787218\) & \(0.9794548728\) & \(2.1318775314\)\\
\(f_{A_1;A_3C_3}\) & \(0.1009160073\) & \(-2.4469912894\) & \(-0.5138117107\) & \(0.8511458052\) & \(-1.0439320521\) & \(2.6601120646\)\\
\(f_{A_1;A_3C_5}\) & \(0.2029558943\) & \(-1.9188886575\) & \(0.5181302757\) & \(-1.7051858713\) & \(-0.5515983584\) & \(2.4907234616\)\\
\(f_{A_1;A_3C_6}\) & \(0.0892992557\) & \(-2.0088953382\) & \(1.6910168478\) & \(0.2161503092\) & \(-0.8941437934\) & \(2.5979582167\)\\
\(f_{A_1;A_3C_8}\) & \(0.0203082820\) & \(-2.2307165937\) & \(1.4451018735\) & \(-1.1220765541\) & \(1.2922878237\) & \(2.9989818247\)\\
\(f_{A_1;C_1C_8}\) & \(-1.3594763480\) & \(-0.9044884080\) & \(0.8225814255\) & \(-2.1486348134\) & \(1.4183217709\) & \(2.9951934211\)\\
\(f_{A_1;C_3C_4}\) & \(-0.9219069136\) & \(-1.6128628225\) & \(-1.9392352397\) & \(0.1560596384\) & \(-0.9493471835\) & \(2.6716075389\)\\
\(f_{A_1;C_6C_7}\) & \(-1.3957707548\) & \(-0.9150398297\) & \(1.2417856581\) & \(1.6586404139\) & \(-0.7939195693\) & \(2.5901547861\)\\
\(f_{A_2;A_1C_1}\) & \(-1.4969830226\) & \(0.7637207895\) & \(-1.7382752965\) & \(-1.5941914265\) & \(0.3596869925\) & \(2.7416509000\)\\
\(f_{A_2;A_1C_2}\) & \(-1.9085205882\) & \(0.9875385814\) & \(-1.7398268802\) & \(0.6609778406\) & \(0.3208328078\) & \(2.6803929292\)\\
\(f_{A_2;A_1C_4}\) & \(-1.4872587450\) & \(0.5855326701\) & \(-1.7743708939\) & \(-0.6320322964\) & \(-1.4524888746\) & \(2.6855852074\)\\
\(f_{A_2;A_1C_5}\) & \(-1.5882964752\) & \(1.1351617075\) & \(0.0274193651\) & \(-2.1332967976\) & \(-0.6635499190\) & \(2.7934357588\)\\
\(f_{A_2;A_1C_7}\) & \(-1.8892599491\) & \(1.4501856508\) & \(0.5124736404\) & \(0.4357803540\) & \(0.3063297489\) & \(2.2844503149\)\\
\(f_{A_2;A_3C_1}\) & \(0.8433278348\) & \(1.7345923181\) & \(-0.5138657261\) & \(-1.9152224056\) & \(0.4794322277\) & \(2.6233570964\)\\
\(f_{A_2;A_3C_3}\) & \(1.0121321439\) & \(1.6870763579\) & \(-1.0418114686\) & \(0.6375054572\) & \(-1.5193598000\) & \(2.5828040305\)\\
\(f_{A_2;A_3C_4}\) & \(0.9873751294\) & \(1.2621061279\) & \(-2.0740087802\) & \(-0.9691410479\) & \(-0.7855049274\) & \(2.7249929003\)\\
\(f_{A_2;A_3C_6}\) & \(0.9289319391\) & \(1.9874993427\) & \(0.5505808281\) & \(-0.7142985971\) & \(-1.3722673902\) & \(2.5513815295\)\\
\(f_{A_2;A_3C_7}\) & \(0.6995112400\) & \(1.9887414718\) & \(0.1900176091\) & \(0.2389423645\) & \(0.9016524654\) & \(2.0858058228\)\\
\(f_{A_2;C_2C_3}\) & \(0.1662444825\) & \(1.8711764325\) & \(-1.8130146922\) & \(1.5881030024\) & \(0.2914181724\) & \(2.9022330095\)\\
\(f_{A_2;C_5C_6}\) & \(-0.6148047792\) & \(2.3998025651\) & \(1.7062064708\) & \(-0.6741401527\) & \(-0.6415157623\) & \(2.9856632777\)\\
\(f_{A_3;A_1C_2}\) & \(1.0326514929\) & \(-1.6402181110\) & \(-0.5353257658\) & \(0.6381320481\) & \(1.2132383338\) & \(2.2186522933\)\\
\(f_{A_3;A_1C_3}\) & \(1.3230478720\) & \(-1.8690783802\) & \(-0.8164604490\) & \(0.5324899380\) & \(-0.8828613966\) & \(2.4440759210\)\\
\(f_{A_3;A_1C_5}\) & \(1.5307547761\) & \(-1.8259935055\) & \(1.2287316271\) & \(-1.6675888754\) & \(-0.4483039747\) & \(3.0280477192\)\\
\(f_{A_3;A_1C_6}\) & \(1.3410144539\) & \(-1.6091422470\) & \(1.2397281544\) & \(0.3142958149\) & \(-1.1514869100\) & \(2.5197794341\)\\
\(f_{A_3;A_1C_8}\) & \(1.3141127413\) & \(-1.7501961829\) & \(1.2277952232\) & \(-0.9755240443\) & \(1.4801051638\) & \(2.9051537898\)\\
\(f_{A_3;A_2C_1}\) & \(1.8192636829\) & \(0.9532838026\) & \(-0.1691278544\) & \(-1.6316330506\) & \(0.9671728129\) & \(2.6162423911\)\\
\(f_{A_3;A_2C_3}\) & \(1.9777522113\) & \(0.7936675248\) & \(-0.7987146740\) & \(0.9101964278\) & \(-1.1441743609\) & \(2.5133542498\)\\
\(f_{A_3;A_2C_4}\) & \(2.0172508175\) & \(0.3743301413\) & \(-1.7583642176\) & \(-1.0209695115\) & \(-0.5603847252\) & \(2.7672510585\)\\
\(f_{A_3;A_2C_6}\) & \(1.8659269357\) & \(1.1087959056\) & \(0.9812123856\) & \(-0.6134856084\) & \(-0.9655292628\) & \(2.4459151221\)\\
\(f_{A_3;A_2C_7}\) & \(1.7115848012\) & \(1.3061338723\) & \(0.2147518691\) & \(0.9517713451\) & \(0.7237548307\) & \(2.2608220491\)\\
\(f_{A_3;C_1C_2}\) & \(2.0490220457\) & \(0.2942890139\) & \(-1.5058512751\) & \(-0.5322052891\) & \(2.0919179286\) & \(3.1956295969\)\\
\(f_{A_3;C_4C_5}\) & \(1.6068283732\) & \(-0.7226994596\) & \(-0.4856233484\) & \(-2.1175759699\) & \(-0.4569873531\) & \(2.6519780157\)\\
\(f_{A_3;C_7C_8}\) & \(1.5570880048\) & \(0.4934813209\) & \(1.8437219566\) & \(-0.0777380106\) & \(1.0894171234\) & \(2.5020452412\)\\
\(f_{C_1;A_1A_2}\) & \(-0.5979316020\) & \(-0.2335452976\) & \(-1.8876380457\) & \(-1.9515103523\) & \(0.5883576967\) & \(2.6701686150\)\\
\(f_{C_1;A_1C_8}\) & \(-0.2510874831\) & \(-0.3853439457\) & \(0.2567668615\) & \(-2.5167818161\) & \(1.6573265730\) & \(2.8910873701\)\\
\(f_{C_1;A_2A_3}\) & \(1.0370090062\) & \(0.7126529966\) & \(-0.8447663159\) & \(-2.6847403824\) & \(1.1788121385\) & \(3.1455239753\)\\
\(f_{C_1;A_3C_2}\) & \(1.0762033885\) & \(0.6092703719\) & \(-1.0940473079\) & \(-1.1270128447\) & \(2.7507003215\) & \(3.2500575137\)\\
\(f_{C_1;C_2C_8}\) & \(-1.1793129119\) & \(0.6904935352\) & \(-0.8754107573\) & \(-1.1915406513\) & \(2.3404640589\) & \(2.9208638097\)\\
\(f_{C_2;A_1A_2}\) & \(-1.2967742763\) & \(0.0923544744\) & \(-2.3481965695\) & \(0.5618663353\) & \(0.9900619272\) & \(2.7386303876\)\\
\(f_{C_2;A_1A_3}\) & \(0.4132492207\) & \(-1.3391999860\) & \(-1.8072025980\) & \(0.6725542550\) & \(1.3784267770\) & \(2.5656582697\)\\
\(f_{C_2;A_2C_3}\) & \(-0.5360307730\) & \(1.2577742865\) & \(-1.5517708363\) & \(2.3010043432\) & \(1.0657938028\) & \(3.1157431362\)\\
\(f_{C_2;A_3C_1}\) & \(1.0866908229\) & \(0.7928699053\) & \(-1.4763081903\) & \(0.2191999154\) & \(2.6053328464\) & \(3.1344590199\)\\
\(f_{C_2;C_1C_8}\) & \(-0.8631049359\) & \(1.0695888346\) & \(-0.7119178111\) & \(0.1132594380\) & \(2.5416343490\) & \(2.8050900591\)\\
\(f_{C_2;C_3C_7}\) & \(0.1329483816\) & \(-1.0380129024\) & \(-0.9341905188\) & \(2.6495833713\) & \(1.5653771679\) & \(3.2308753771\)\\
\(f_{C_2;C_7C_8}\) & \(-0.0805218482\) & \(-0.9585091372\) & \(-0.0133075551\) & \(1.5100313544\) & \(3.1770855017\) & \(3.5070596803\)\\
\(f_{C_3;A_1A_3}\) & \(0.6875257104\) & \(-1.9329284837\) & \(-1.2281255910\) & \(1.8515187860\) & \(-0.9682887816\) & \(3.0137852565\)\\
\(f_{C_3;A_1C_4}\) & \(-0.8463173214\) & \(-0.5864761340\) & \(-1.9791043322\) & \(1.0248908164\) & \(-1.3804439180\) & \(2.6330757335\)\\
\(f_{C_3;A_2A_3}\) & \(1.1037928316\) & \(0.7366436105\) & \(-1.7752922827\) & \(1.4336274124\) & \(-1.5100446840\) & \(2.8719658108\)\\
\(f_{C_3;A_2C_2}\) & \(0.2692732290\) & \(1.1475038312\) & \(-1.3911717896\) & \(2.7004564259\) & \(-0.0085361512\) & \(3.1011562096\)\\
\(f_{C_3;C_2C_7}\) & \(0.6169435550\) & \(-0.7331244399\) & \(-0.6983591222\) & \(3.0240726390\) & \(0.3440429096\) & \(3.1095300461\)\\
\(f_{C_3;C_4C_6}\) & \(-0.6356071782\) & \(-0.7626939955\) & \(-0.1535912827\) & \(1.6837616882\) & \(-2.7570908116\) & \(3.2320105297\)\\
\(f_{C_3;C_6C_7}\) & \(0.9723189798\) & \(0.3635610910\) & \(0.4697297754\) & \(2.6780184710\) & \(-1.2074271884\) & \(2.9879575421\)\\
\(f_{C_4;A_1A_2}\) & \(-0.9144922379\) & \(-0.1736068966\) & \(-1.7171864684\) & \(-1.5997737578\) & \(-1.9960088930\) & \(3.0591653032\)\\
\(f_{C_4;A_1C_3}\) & \(-0.1289984414\) & \(-1.1829113580\) & \(-2.1971098601\) & \(0.1318400072\) & \(-2.0925712463\) & \(3.1047460161\)\\
\(f_{C_4;A_2A_3}\) & \(1.1847364642\) & \(0.0802132983\) & \(-2.2932812317\) & \(-0.8410374320\) & \(-1.5277520804\) & \(2.9513630498\)\\
\(f_{C_4;A_3C_5}\) & \(0.7780228276\) & \(-0.6190260445\) & \(-0.7510760043\) & \(-2.2061948143\) & \(-1.5746520167\) & \(2.8105964599\)\\
\(f_{C_4;C_3C_6}\) & \(-0.4548958747\) & \(-1.0891407759\) & \(-0.2752460345\) & \(0.4119181458\) & \(-3.2481750671\) & \(3.3450315531\)\\
\(f_{C_4;C_5C_6}\) & \(-0.7897330148\) & \(0.2188812590\) & \(0.0843787383\) & \(-1.3127570715\) & \(-2.8536965275\) & \(3.0895990050\)\\
\(f_{C_5;A_1A_2}\) & \(-1.0521785948\) & \(0.3665325126\) & \(0.8301936585\) & \(-2.5881203505\) & \(-0.1427401480\) & \(2.7657528970\)\\
\(f_{C_5;A_1A_3}\) & \(0.7223668362\) & \(-1.1326872960\) & \(1.2963036463\) & \(-2.0800897615\) & \(-1.2879126310\) & \(2.9104449593\)\\
\(f_{C_5;A_2C_6}\) & \(-0.3016534858\) & \(1.0834821878\) & \(2.1442292429\) & \(-1.0815719701\) & \(-0.3992203478\) & \(2.4883372632\)\\
\(f_{C_5;A_3C_4}\) & \(1.0306416816\) & \(0.2869389171\) & \(-0.0283419354\) & \(-2.6130339320\) & \(-0.9206298151\) & \(2.7965809607\)\\
\(f_{C_5;C_4C_6}\) & \(-0.6137057368\) & \(0.9455290573\) & \(1.1222124027\) & \(-1.7243044467\) & \(-2.3984334087\) & \(3.2024567518\)\\
\(f_{C_6;A_1A_3}\) & \(0.5529297195\) & \(-1.0507349208\) & \(2.0470048684\) & \(0.9567090382\) & \(-1.3453738005\) & \(2.7065341170\)\\
\(f_{C_6;A_1C_7}\) & \(-1.1996010402\) & \(-0.1443998175\) & \(2.5222288809\) & \(1.5167056730\) & \(-1.1194345520\) & \(3.2210343533\)\\
\(f_{C_6;A_2A_3}\) & \(1.2356449846\) & \(1.4136872579\) & \(1.3217933096\) & \(0.1354685409\) & \(-1.9659840683\) & \(2.8558558834\)\\
\(f_{C_6;A_2C_5}\) & \(-0.9639099323\) & \(1.4664969605\) & \(2.2436660341\) & \(-0.2115176016\) & \(-1.3516755851\) & \(2.9975889559\)\\
\(f_{C_6;C_3C_4}\) & \(0.2193961534\) & \(-0.7988185244\) & \(0.7312013631\) & \(1.0606398152\) & \(-3.0720657148\) & \(3.2838157248\)\\
\(f_{C_6;C_3C_7}\) & \(0.6691489609\) & \(1.1612667342\) & \(1.3415890011\) & \(2.1276511039\) & \(-1.6964631524\) & \(3.1624433995\)\\
\(f_{C_6;C_4C_5}\) & \(-1.1149298647\) & \(0.1656000043\) & \(1.2589873399\) & \(-0.6511592257\) & \(-2.5559871144\) & \(2.9686056580\)\\
\(f_{C_7;A_1A_2}\) & \(-1.7102669574\) & \(0.8276902179\) & \(0.7888408942\) & \(1.2031382036\) & \(1.2914603947\) & \(2.5194772488\)\\
\(f_{C_7;A_1C_6}\) & \(-1.5563465195\) & \(0.3241378792\) & \(1.7083931921\) & \(2.2190867780\) & \(-0.4015062551\) & \(3.0873031198\)\\
\(f_{C_7;A_2A_3}\) & \(0.8595291322\) & \(1.8390603402\) & \(1.1988358740\) & \(1.3759534735\) & \(0.8715458952\) & \(2.6853269681\)\\
\(f_{C_7;A_3C_8}\) & \(0.5899337152\) & \(0.2783245165\) & \(2.4011212864\) & \(0.7690040325\) & \(1.2994938360\) & \(2.7332985912\)\\
\(f_{C_7;C_2C_3}\) & \(0.0668878578\) & \(-0.8399705895\) & \(0.3376813096\) & \(3.0405953039\) & \(1.1340619323\) & \(3.2179759962\)\\
\(f_{C_7;C_2C_8}\) & \(0.1686652769\) & \(-0.7506863196\) & \(1.1974083223\) & \(1.7263001253\) & \(2.4814626706\) & \(3.1880297559\)\\
\(f_{C_7;C_3C_6}\) & \(0.2823573877\) & \(1.2516329589\) & \(1.0111375989\) & \(2.7585984224\) & \(-0.5390087614\) & \(3.0933971137\)\\
\(f_{C_8;A_1A_3}\) & \(0.3986770570\) & \(-1.3736780717\) & \(2.0319039251\) & \(-0.7037196705\) & \(2.0865523158\) & \(3.1659896308\)\\
\(f_{C_8;A_1C_1}\) & \(-0.6217032321\) & \(-0.2235915763\) & \(1.4550931177\) & \(-2.0487403989\) & \(2.2310907786\) & \(3.2755010865\)\\
\(f_{C_8;A_3C_7}\) & \(0.6560085557\) & \(0.5863386863\) & \(2.0550932316\) & \(-0.1901861689\) & \(2.2751368511\) & \(3.0347927344\)\\
\(f_{C_8;C_1C_2}\) & \(-1.2520743088\) & \(0.5216221912\) & \(0.3702078053\) & \(-0.5054999742\) & \(2.4269994985\) & \(2.6688368990\)\\
\(f_{C_8;C_2C_7}\) & \(-0.7746559993\) & \(-0.9674041009\) & \(1.0421754060\) & \(0.8767458888\) & \(3.1147930684\) & \(3.4774576059\)\\
\end{longtable}

\endgroup

\bibliographystyle{amsplain}

\end{document}